\documentclass[10pt,draft,reqno]{amsart}
\usepackage{color}

     \makeatletter
     \def\section{\@startsection{section}{1}%
     \z@{.7\linespacing\@plus\linespacing}{.5\linespacing}%
     {\bfseries%\normalfont\scshape
     \centering
     }}
     \def\@secnumfont{\bfseries}
     \makeatother
\newtheorem{theorem}{Theorem}[section]
\newtheorem{lemma}[theorem]{Lemma}
\newtheorem{proposition}[theorem]{Proposition}
\newtheorem{corollary}[theorem]{Corollary}
\theoremstyle{definition}

\theoremstyle{remark}
\newtheorem{remark}[theorem]{Remark}
\numberwithin{equation}{section}
\newtheorem{assumption}[theorem]{Assumption}%[section]
\newtheorem{claim}{Claim}

\newcommand{\Rd}{\mathbb{R}^{d}}
\newcommand{\hs}{\mathbb{R}_+ \times \mathbb{R}^{d}}

\newcommand{\BN}{\mathbb{N}}
\newcommand{\BZ}{\mathbb{Z}}

\newcommand{\BR}{\mathbb{R}}

\newcommand{\CB}{\mathcal{B}}
\newcommand{\CS}{\mathcal{S}}
\newcommand{\CF}{\mathcal{F}}

\newcommand{\CM}{\mathcal{M}}

\newcommand{\CA}{\mathcal{A}}

\newcommand{\CD}{\mathcal{D}}

\newcommand{\CK}{\mathcal{K}}

\newcommand{\bfP}{\mathbf{P}}
\newcommand{\bfE}{\mathbf{E}}
\newcommand{\bfi}{\mathbf{1}}

\begin{document}

\title[SDEs with Singular Drift and Driven by Stable Processes]{On Weak Solutions of SDEs with Singular Time-Dependent Drift and Driven by Stable Processes}
\author[P. Jin]{Peng Jin}
\address{Peng Jin:  Fakult\"at f\"ur Mathematik und Naturwissenschaften, Bergische Universit\"at Wuppertal, 42119 Wuppertal, Germany}
\email{jin@uni-wuppertal.de}



\begin{abstract}
Let $d\ge2$. In this paper, we study weak solutions for the following type of stochastic differential equation
\[
\begin{cases}dX_{t}=dS_{t}+b(s+t, X_{t})dt, & \ t\ge 0,\\
   \   X_{0}=x,
    \end{cases}
\]
where $(s,x)\in \hs$ is the initial starting point, $b: \hs \to \Rd$ is measurable, and $S=(S_{t})_{t \ge 0}$ is a $d$-dimensional $\alpha$-stable process with index $\alpha \in (1,2)$. We show that if the $\alpha$-stable process $S$ is non-degenerate and $b \in L_{loc}^{\infty}(\mathbb{R}_{+};L^{\infty}(\mathbb{R}^{d}))+ L_{loc}^{q}(\mathbb{R}_{+};L^{p}(\mathbb{R}^{d}))$ for some $p,q>0$ with $d/ p+\alpha/q <\alpha-1$, then the above SDE has a unique weak solution for every starting point $(s,x)\in \hs$.
\end{abstract}
%\keywords{Stochastic differential equations \and Singular drift \and Stable process  \and Weak solutions \and Martingale problem \and Resolvent}
%\PACS{PACS code1 \and PACS code2 \and more}

\maketitle

%

%%%%%% Section 1: Introduction
\section{Introduction}
In this paper, we study the following type of stochastic differential equation
\begin{equation} \label{WUSeqsect1}
\begin{cases}dX_{t}=dS_{t}+b(s+t, X_{t})dt, & \ t\ge 0,\\
   \   X_{0}=x,
    \end{cases}
\end{equation}
where $(s,x)\in \hs$ is the initial starting point, $b: \hs \to \Rd$ is measurable, and $S=(S_{t})_{t \ge 0}$ is a $d$-dimensional $\alpha$-stable process. The equation (\ref{WUSeqsect1}) is a shorthand way of writing
\begin{equation} \label{WUSeqsect115}
X_{t}=x+S_{t}+\int_0^tb(s+u, X_{u})du,  \quad t\ge 0.
\end{equation}
Since the drift $b$ is not assumed to be bounded, solutions of (\ref{WUSeqsect1}) are supposed to fulfill the integrability conditions
\begin{equation}\label{WUSeqsect12}
\int_0^t|b(s+u,X_u)|du<\infty \quad \mbox{a.s.}, \quad \forall  t \ge 0,
\end{equation}
such that the integral in (\ref{WUSeqsect115}) makes sense.

The classical results tell us, that if $b$ is of linear growth and globally Lipschitz in the space variable, then a unique strong solution to (\ref{WUSeqsect1}) exists. However, it turns out that the Lipschitz-continuity on the drift is far from being necessary; this is well demonstrated in the special case where $\alpha=2$, that is, when $S$ is a $d$-dimensional Brownian motion. Indeed, there is an extensive literature devoted to the study of Brownian motion (or more generally, diffusions) with singular drift, see, e.g., \cite{MR0336813,MR532447,MR1104660,MR1246978,MR1964949,MR2117951,MR2593276,MR2820071}, and many others. In particular, it was shown in \cite{MR2117951} that if $S$ is a Brownian motion and there exist $p,q>0$ with $d/ p+2/q <1$ such that $b\in L_{loc}^{q}(\mathbb{R}_{+};L^{p}(\mathbb{R}^{d}))$, namely
\begin{equation}\label{WUSeqsect105}
\int_0^T\Big(\int_{\Rd}|b(t,x)|^{p}dx\Big)^{q/p}dt<\infty, \quad \forall  T >0,
\end{equation}
then strong existence and uniqueness hold for (\ref{WUSeqsect1}). Regarding weak solutions to (\ref{WUSeqsect1}) in the Brownian case, the condition (\ref{WUSeqsect105}) can be relaxed, see, e.g., \cite{MR1964949,jin2009stochastic}, where weak existence and uniqueness for (\ref{WUSeqsect1}) (in the case $\alpha=2$) were shown for some Kato-class drifts.

There is now a growing interest to study (\ref{WUSeqsect1}) for the case where $\alpha \in (0,2)$. Earlier works in this direction include \cite{MR736974,MR1341116}, which primarily concentrated on weak solutions, or equivalently, solutions to the corresponding martingale-problem. Recently, weak existence and uniqueness of rotationally symmetric $\alpha$-stable ($1<\alpha<2$) processes with (time-independent) singular drift belonging to the Kato-class $\CK_{d,\alpha-1}$ were obtained in \cite{chen2013uniqueness,MR3192504}. Compared to weak solutions, one needs generally more regularity on the drift to obtain  strong solutions to (\ref{WUSeqsect1}), as seen in the diffusion case; this is also the case when $\alpha \in (0,2)$. Priola \cite{MR2945756} proved that if the stable process $S$ is symmetric, non-degenerate and with index $\alpha \in (1,2)$, and $b$ is time-independent and belongs to $ C^{\beta}_b(\Rd)$ with $\beta>1-\alpha/2$, then pathwise uniqueness holds for (\ref{WUSeqsect1}). Afterwards, similar results were obtained by Zhang \cite{MR3127913} where it was shown that if $S$ is as in \cite{MR2945756}, $b$ is locally bounded and there exist some $\beta \in (1-\alpha/2,1)$ and $p>2d/\alpha$ such that for any $T,R >0$,
\[
\sup_{t\in [0,T]} \int_{\{|x|\le R\}}\int_{\{|y|\le R\}}\frac{|b(t,x)-b(t,y)|^p}{|x-y|^{d+\beta p}}dxdy<\infty,
\]
then  a unique strong solution to (\ref{WUSeqsect1}) exists. As an intermediate step to obtain the main result, Zhang \cite{MR3127913} also obtained the following result  on weak existence: if the stable process $S$ is symmetric, non-degenerate and with index $\alpha \in (1,2)$, and there exist $p,q>0$ such that
\begin{equation}\label{WUSeqsect122}
d/ p+\alpha/q <\alpha-1 \quad \mbox{and} \quad b\in L_{loc}^{\infty}(\mathbb{R}_{+};L^{\infty}(\mathbb{R}^{d}))+ L_{loc}^{q}(\mathbb{R}_{+};L^{p}(\mathbb{R}^{d})),
\end{equation}
then weak solutions to (\ref{WUSeqsect1}) exist. Very recently, the results of \cite{MR2945756} have been extended in \cite{chen2015arxiv1501} to a larger class of L\'evy processes, including the rotationally symmetric $\alpha$-stable process with index $0<\alpha\le 1$. For SDEs with irregular drift and driven by other types of L\'evy noises, see also \cite{MR3177338,MR2359059}.

In this paper, we study weak solutions to (\ref{WUSeqsect1}) for the case $1<\alpha<2$. We are mainly interested in the weak uniqueness problem. Under mild assumptions on the stable process, we will prove that the condition (\ref{WUSeqsect122}) on $b$ implies weak uniqueness for (\ref{WUSeqsect1}). More precisely, the main result of the present paper is as follows:

\begin{theorem}\label{WUSmainthm}Let $d\ge2$ and $1<\alpha<2$. Assume that the $\alpha$-stable process $S$ is non-degenerate, that is, Assumption \ref{WUSass21} (see Section 2) is satisfied.  Assume that $b$ is such that (\ref{WUSeqsect122}) holds. Then the SDE (\ref{WUSeqsect1}) has a unique weak solution for every starting point $(s,x)\in \hs$.
\end{theorem}
\noindent

We remark that the stable process $S$ considered in Theorem \ref{WUSmainthm} is not necessarily symmetric. Therefore, Theorem \ref{WUSmainthm} also extends the above mentioned result of \cite{MR3127913} on weak existence for (\ref{WUSeqsect1}).

We now briefly discuss our strategy to prove Theorem \ref{WUSmainthm}. Essentially, our proof of Theorem \ref{WUSmainthm} is based on some perturbation arguments. Under the assumptions of Theorem \ref{WUSmainthm}, we will see that the time-space resolvent of the solution $X$ to (\ref{WUSeqsect1}) can be explicitly expressed in terms of a series, of which the main term is the time-space resolvent of $S$. This enables us to obtain  weak uniqueness for (\ref{WUSeqsect1}). On the other hand, some estimates on the time-space resolvent of $X$ can be established and used as substitutes of the Krylov-estimates obtained in \cite{MR3127913}; as a consequence, we can adapt the proof of weak existence in \cite{MR3127913}  to make it work also  for our case. We would like to point out that the perturbation on the resolvent of $S$ that we will do   in this paper depends mainly on the scaling property of the heat kernel of $S$, rather than exact estimates of that. As shown in \cite{MR2286060}, sharp heat kernel estimates are actually not available in our case, since $S$ is merely assumed to be non-degenerate. Therefore, we can not carry out the same perturbation on the heat kernel of $S$ as done in \cite{chen2013uniqueness,MR3192504}, where a rotationally symmetric $\alpha$-stable process is considered for which sharp heat kernel estimates are known.

The rest of the paper is organized as follows. In Section 2 we recall some basic facts on $\alpha$-stable processes and the definition of the martingale problem for non-local generators. In Section 3 we establish some estimates on the time-space resolvent of non-degenerate $\alpha$-stable processes and obtain a solvability result on the corresponding resolvent equation. In Section 4 we prove the local existence and uniqueness of weak solutions to (\ref{WUSeqsect1}), under slightly stronger assumptions than those in Theorem \ref{WUSmainthm}. Finally, we prove Theorem \ref{WUSmainthm} in Section 5.

%%%%%% Section 2: Preliminaries
\section{Preliminaries}

Throughout this paper, we assume that dimension $d\ge 2$. The inner product of $x$ and $y$ in $\mathbb{R}^{d}$ is written as $x\cdot y$. We use $|v|$ to denote the Euclidean norm of a vector $v\in \BR^m$, $m\in \BN$.  For a bounded function $g:\hs\to \BR^m$ we write $\|g\|:=\sup_{(s,x) \in \hs} |g(s,x)|$. Let $\mathbb{S}^{d-1}:=\{x \in \mathbb{R}^{d}:|x|=1\}$ be the unitary sphere.

Let $\alpha \in (1,2)$ be fixed throughout this paper. A $d$-dimensional $\alpha$-stable process $S=(S_{t})_{t\geq 0}$ is a L\'{e}vy process with a particular form of characteristic exponent $\psi$, namely
\begin{center}
$\bfE\big[e^{iS_{t}\cdot u} \big ]=e^{-t\psi(u)},\quad  u\in\mathbb{R}^{d}$,
\begin{equation}\label{WUSeqsect21}
 \psi(u)=-\int_{\mathbb{R}^{d}\setminus\{0\}}\Big(e^{iu\cdot y}-1-iu\cdot y\Big)\nu(dy)-iu\cdot \gamma,
\end{equation}
\end{center}
where $\gamma \in \Rd$ and the L\'{e}vy measure $\nu$ is given by

\[\nu(B)=\int_{\mathbb{S}^{d-1}}\mu(d\xi)\int_{0}^{\infty}\mathbf{1}_{B}(r\xi)\frac{dr}{r^{1+\alpha}},\quad\ B\in\mathcal{B}(\mathbb{R}^{d}). \]
Here, $\gamma$ is called the center of $S$ and $\gamma=\bfE[S_{1}]$; the measure $\mu$ is a finite measure on the unitary sphere $\mathbb{S}^{d-1}$ and is called the spectral measure of the $\alpha$-stable process $S$.
It's worth noting that the first term on the right-hand side of (\ref{WUSeqsect21}) is a homogeneous function (with variable $u$) of index $\alpha$. As a consequence,
\begin{equation}\label{WUSeqsect215}
\psi(\rho u)+i(\rho u\cdot \gamma)= \rho^{\alpha}(\psi( u)+i( u\cdot \gamma)), \quad \forall \rho>0.
\end{equation}

Throughout this paper, we assume $S$ to be non-degenerate, that is, the spectral measure $\mu$ of $S$ satisfies the following condition.

\begin{assumption}\label{WUSass21}The support of $\mu$ is not contained in a proper linear subspace of $\mathbb{R}^{d}$.
\end{assumption}

The infinitesimal generator $A$ of the process $S$ is given by

\begin{equation}\label{WUSgeneratorA}
Af(x)=\int_{\mathbb{R}^{d}\setminus\{0\}}\big(f(x+z)-f(x)- z\cdot\nabla f(x)\big)\nu(dz)+\sum_{i=1}^d\gamma_i \partial_{x_i}f(x),\quad  f\in C_{b}^{2}(\mathbb{R}^{d}),
\end{equation}
where $C_{b}^{2}(\mathbb{R}^{d})$ denotes the class of $C^{2}$ functions such that the function and its first and second order partial derivatives are bounded. Note that $C_{b}^{2}(\mathbb{R}^{d})$ is a Banach space endowed with the norm
\[
\|f\|_{C_{b}^{2}(\mathbb{R}^{d})}:=\|f\|+\sum_{i=1}^d \|\partial_{i}f\|+\sum_{i,j=1}^d \|\partial^2_{ij}f\|, \quad f \in C_{b}^{2}(\mathbb{R}^{d}),
\]
where $\partial_{i}f(x):=\partial_{x_i}f(x)$ and $\partial^2_{ij}f(x):=\partial^2_{x_ix_j}f(x)$ for $x \in \Rd$. For a function $f(x,y)$ with two variables $x,y \in \Rd$, we also use the notation $A_xf(x,y)$ to indicate that $A$ is operating on the function $f(\cdot,y)$ with $y$ being considered as fixed.

Let
\begin{equation}\label{WUSgeneratorlt}
L_{t}:=A+b(t,\cdot)\cdot\nabla,\end{equation}
where $\nabla$ is the gradient operator with respect to the spatial variable.

Let $D=D\big([0,\infty)\big)$, the set of paths that are right continuous with left limits, endowed with the Skorokhod topology. Set $X_{t}(\omega)=\omega(t)$ for $\omega\in\Omega$ and let $\mathcal{D}=\sigma(X_{t}: 0\le t<\infty)$ and $\mathcal{F}_{t}:=\sigma(X_{r}:0\le r\le t)$. A probability measure $\mathbf{P}$ on $(D,\mathcal{D})$ is called a solution to the martingale problem for $L_{t}$ starting from $(s,x)$, if
\begin{equation}\label{WUSwsintegrcondi}
    \mathbf{P}(X_{t}=x,\ \forall t \le s)=1, \quad \mathbf{P}\Big(\int_s^t|b(u,X_u)|du<\infty, \ \forall t \ge s\Big)=1, \end{equation}
 and under the measure $\mathbf{P}$,
\begin{equation}\label{WUSeq2defimp}
    f(X_{t})-\int^{t}_{s}L_{u}f(X_{u})du
\end{equation}
is an $\mathcal{F}_{t}$-martingale after time $s$ for all $f \in C^{2}_{b}(\mathbb{R}^{d})$.

%%%%%% Section 3: Some Analytical Results
\section{Some Analytical Results}

We first recall that $\psi$ is the characteristic exponent of the stable process $(S_{t})_{t\geq 0}$. According to Assumption \ref{WUSass21} and \cite[Prop.~24.20]{MR1739520}, there exists some constant $c>0$ such that
\begin{equation}\label{WUSsect31}
\big| \bfE[e^{iS_{t}\cdot u}]\big|=\big|e^{-t\psi(u)}\big|\le e^{-ct|u|^{\alpha}}, \quad u \in \Rd.
\end{equation}
By the inversion formula of Fourier transform, the law of $S_{t}$ has a density $p_{t}\in L^{1}(\mathbb{R}^{d})\cap C_{b}(\mathbb{R}^{d})$ that is given by
\begin{equation}\label{WUSsect312}
p_{t}(x)=\frac{1}{(2\pi)^{d}}\int_{\mathbb{R}^{d}}e^{-iu\cdot x}e^{-t\psi(u)}du, \quad x\in\mathbb{R}^{d},\ t>0.
\end{equation}
According to \cite[p.~2856,~(2.3)]{MR2286060}, we have the following scaling property for $p_t$:
\begin{equation}\label{WUSsect315}
p_{t}(x)=t^{-d/\alpha}p_1(t^{-1/\alpha}x+(1-t^{1-1/\alpha})\gamma), \quad x\in\mathbb{R}^{d},\ t>0.
\end{equation}

The following result is a slight extension of \cite[Lemma~3.1]{MR2945756}.

%%%%%%%%%%%%%%%%%%%%
%  Lemma 3.1
%%%%%%%%%%%%%%%%%%%%
\begin{lemma}\label{WUScoro31}Let $p\geq 1$ be arbitrary. Then for each $t>0$, the density $p_{t}$ of $S_t$ and all its derivatives $D^kp_{t}$ belong to  $ C_b^{\infty}(\mathbb{R}^{d}) \cap L^{p}(\mathbb{R}^{d})$, where $k=(k_1,\cdots,k_d)$ is a multi-index with $k_i\in\BZ_+, \ i=1,\cdots,d, $ and
\[
D^k:=\frac{\partial^{|k|}}{(\partial x_1)^{k_1}\cdots (\partial x_d)^{k_d}} \quad \mbox{with} \quad |k|=k_1+\cdots+k_d .
\]
\end{lemma}

\begin{proof} We follow the proof of \cite[Lemma 3.1]{MR2945756}. We will only prove the assertion for $p_{t}$, since the cases for the derivatives $D^{k}p_{t}$ are similar. By the scaling property (\ref{WUSsect315}), it suffices to consider $t=1.$ As shown in the proof of \cite[Lemma 3.1]{MR2945756}, the characteristic exponent $\psi$ can be written as the sum of $\psi_{1}$ and $\psi_{2}$, where
\[
\psi_{1}(u)= -\int_{\{0<|y|\le 1\}}\Big(e^{iu\cdot y}-1-iu\cdot y\Big)\nu(dy)  , \quad \psi_{2}=\psi-\psi_1.
\]
It is easy to see that $\exp(-\psi_2)$ is bounded and is the characteristic function of an infinitely divisible probability measure $m$ on $\mathbb{R}^{d}$. It follows from (\ref{WUSsect31}) that
\[
\big|e^{-\psi_1(u)}\big|\le c_1e^{-c|u|^{\alpha}}, \quad u \in \Rd,
\]
for some constant $c_1>0$. We can easily check that $\psi_{1}\in C^{\infty}(\Rd)$ and that $\exp(-\psi_{1})$ belongs to the Schwartz space $\mathcal{S}(\mathbb{R}^{d})$. Since the Fourier transform is a one-to-one map of $\mathcal{S}(\mathbb{R}^{d})$ onto itself, we can find $f\in \mathcal{S}(\mathbb{R}^{d})$ with $ \hat{f}$=$\exp(-\psi_{1})$, where $ \hat{f}$ denotes the Fourier transform of $f$. In particular, we have $f\in L^{p}(\mathbb{R}^{d})$ for all $p \ge 1$. Let $ f*m$ be the convolution of $f$ and $\gamma.$ We have
$$
\ \widehat{f*m}=\hat{f}\hat{m}=e^{-\psi_{1}-\psi_{2}}=e^{-\psi}=\ \hat{p}_1,
$$
which implies $ p_{1}=f*m$. Thus $p_1 \in C_b^{\infty}(\mathbb{R}^{d})$. By Young's inequality, we get $p_{1}\in L^{p}(\mathbb{R}^{d}).$
\end{proof}

Recall that the operator $A$ defined in (\ref{WUSgeneratorA}) is the  infinitesimal generator of the process $S$. The following corollary is straightforward.
%%%%%%%%%%%%%%%%%%%%
%  Coro 3.15
%%%%%%%%%%%%%%%%%%%%
\begin{corollary}\label{WUScoro315}For each $t>0$, we have
\begin{equation}\label{WUSeqcoro3151}
Ap_t(x)=-\frac{1}{(2\pi)^d}\int_{\Rd}\psi(-u)e^{-t\psi(u)}e^{-iu\cdot x}du, \quad x \in \Rd.
\end{equation}
\end{corollary}
\begin{proof}
According to Lemma \ref{WUScoro31}, $p_t \in C_b^{\infty}(\mathbb{R}^{d})$. Thus $Ap_t$ is well-defined. For $x \in \Rd$, we have
\begin{align*}
Ap_t(x)=& \int_{\mathbb{R}^{d}\setminus\{0\}}\big(p_t(x+z)-p_t(x)- z\cdot\nabla p_t(x)\big)\nu(dz)+\sum_{i=1}^d\gamma_i \partial_{x_i}p_t(x) \\
=& \int_{\mathbb{R}^{d}\setminus\{0\}}\frac{1}{(2\pi)^{d}}\int_{\mathbb{R}^{d}}\Big(e^{-iu\cdot (x+z)}-e^{-iu\cdot x}+z\cdot iu e^{-iu\cdot x}\Big)e^{-t\psi(u)}du\nu(dz) \\
& \quad +\frac{1}{(2\pi)^{d}}\int_{\mathbb{R}^{d}}(-iu \cdot \gamma  )e^{-iu\cdot x}e^{-t\psi(u)}du.
\end{align*}
By Fubini's theorem,
\begin{align*}
Ap_t(x)=& \frac{1}{(2\pi)^{d}}\int_{\mathbb{R}^{d}}\Big(\int_{\mathbb{R}^{d}\setminus\{0\}}\big(e^{-iu\cdot (x+z)}-e^{-iu\cdot x}+z\cdot iu e^{-iu\cdot x}\big)e^{-t\psi(u)}\nu(dz)\Big)du \\
& \quad +\frac{1}{(2\pi)^{d}}\int_{\mathbb{R}^{d}}(-iu \cdot \gamma  )e^{-iu\cdot x}e^{-t\psi(u)}du \\
=& \frac{1}{(2\pi)^{d}}\int_{\mathbb{R}^{d}}e^{-iu\cdot x}e^{-t\psi(u)}\int_{\mathbb{R}^{d}\setminus\{0\}}\big(e^{-iu\cdot z}-1-iz\cdot(-u) \big)\nu(dz)du \\
& \quad +\frac{1}{(2\pi)^{d}}\int_{\mathbb{R}^{d}}(-iu \cdot \gamma  )e^{-iu\cdot x}e^{-t\psi(u)}du \\
=&\frac{1}{(2\pi)^{d}}\int_{\mathbb{R}^{d}}e^{-iu\cdot x}(-\psi(-u))e^{-t\psi(u)}du.
\end{align*}
\end{proof}

\begin{remark}Let $y \in \Rd$ be fixed. Later in the proof of Proposition \ref{WUSprop31} we will need to calculate $A(p_t(y-\cdot))$. Proceeding as in the proof of Corollary \ref{WUScoro315}, we can easily verify that
\begin{equation}\label{WUSeqremark3155}
A_x\big((p_t(y-x)\big)=-\frac{1}{(2\pi)^d}\int_{\Rd}\psi(u)e^{-t\psi(u)}e^{-iu\cdot (y-x)}du, \quad x \in \Rd.
\end{equation}
\end{remark}

Next, we show that  the the right-hand side of (\ref{WUSeqremark3155}) is an integrable function with respect to the variable $y$.
%%%%%%%%%%%%%%%%%%%%
%  Lemma 3.16
%%%%%%%%%%%%%%%%%%%%

\begin{lemma}\label{WUSlemma316}
Denote by $g(t,x,y)$ the right-hand side of (\ref{WUSeqremark3155}), namely,
\[
g(t,x,y):=-\frac{1}{(2\pi)^d}\int_{\Rd}\psi(u)e^{-t\psi(u)}e^{-iu\cdot (y-x)}du, \quad t>0, \ x,y \in \Rd.
\]
Then $\|g(t,x,\cdot)\|_{L^1(\Rd)}$ is finite and uniformly bounded for $(t,x) \in [\delta,\infty)\times \Rd$, where $\delta>0$ is an arbitrary constant.
\end{lemma}
\begin{proof}
Using (\ref{WUSeqsect215}) and a change of variables $u=t^{-1/\alpha}u'$ for the integral in the definition of $g(t,x,y)$, we obtain
\begin{align*}
g(t,x,y)=&-\frac{t^{-d/\alpha}}{(2\pi)^d}\int_{\Rd}\big(t^{-1}(\psi(u')+iu'\cdot \gamma)-it^{-1/\alpha}u'\cdot \gamma\big) \\
& \qquad \qquad \qquad \times e^{-(\psi(u')+iu'\cdot\gamma)+it^{1-1/\alpha}u'\cdot\gamma}e^{-it^{-1/\alpha}u'\cdot (y-x)}du' \\
=&-\frac{t^{-d/\alpha}}{(2\pi)^d}\int_{\Rd}\big(t^{-1}\psi(u')+iu'\cdot \gamma(t^{-1}-t^{-1/\alpha})\big) \\
& \qquad \qquad \qquad \times e^{-\psi(u')+iu'\cdot\gamma(t^{1-1/\alpha}-1)}e^{-it^{-1/\alpha}u'\cdot (y-x)}du' \\
=& -\frac{t^{-1-d/\alpha}}{(2\pi)^d}\int_{\Rd}\psi(u')e^{-\psi(u')}e^{-iu'\cdot t^{-1/\alpha}(y-x-\gamma t+\gamma t^{1/\alpha})}du'\\
&  \ \ -i\frac{t^{-1-d/\alpha}-t^{-(d+1)/\alpha}}{(2\pi)^d}\int_{\Rd}\gamma \cdot u'e^{-\psi(u')}e^{-iu'\cdot t^{-1/\alpha}(y-x-\gamma t+\gamma t^{1/\alpha})}du'.
\end{align*}
With another change of variables $y'=t^{-1/\alpha}(y-x-\gamma t+\gamma t^{1/\alpha})$, we get
\begin{align}
\int_{\Rd}|g(t,x,y)|dy\le & \frac{t^{-1}}{(2\pi)^d}\int_{\Rd}\Big|\int_{\Rd}\psi(u')e^{-\psi(u')}e^{-iu'\cdot y'}du'\Big|dy' \notag \\
\label{WUSeqlemma316} &\quad +\frac{|t^{-1}-t^{-1/\alpha}|}{(2\pi)^d}\int_{\Rd}\Big|\int_{\Rd}\gamma \cdot u'e^{-\psi(u')}e^{-iu'\cdot y'}du'\Big|dy'.
% & C t^{-1}\int_{\Rd}|A\big(p_1(y'-x)\big)|dy' + (|t^{-1}-t^{-1/\alpha}|)\|f\|_{L^1(\Rd)}
\end{align}
Let $\psi_1$ and $\psi_2$ be as in the proof of Lemma \ref{WUScoro31}. We can further write $\psi_2=\psi_{21}+\psi_{22}$, where
\[
 \psi_{21}(u)=-\int_{\{|y|> 1\}}e^{iu\cdot y}\nu(dy)
\]
and
\[
\psi_{22}(u)=\int_{\{|y|> 1\}}(1+iu\cdot y)\nu(dy)-iu\cdot \gamma.
\]
Then
\begin{align*}
 \psi e^{-\psi}
=& \psi_1 e^{-\psi_1}e^{-\psi_2}+\psi_2e^{-\psi_1}e^{-\psi_2} \\
=& \psi_1 e^{-\psi_1}e^{-\psi_2}-e^{-\psi_1}(-\psi_{21})e^{-\psi_2}+\psi_{22}e^{-\psi_1}e^{-\psi_2}.
\end{align*}
As shown in the proof of Lemma \ref{WUScoro31}, $\exp(-\psi_1) \in \CS(\Rd)$; similarly, we have $\psi_1 \exp(-\psi_1), \psi_{22} \exp(-\psi_1) \in \CS(\Rd)$. Noting that $-\psi_{21}$ and $e^{-\psi_2}$ are both characteristic functions of some finite measures on $\Rd$, we can argue as in the proof of Lemma \ref{WUScoro31} to conclude that
\[
\int_{\Rd}\Big|\int_{\Rd}\psi(u')e^{-\psi(u')}e^{-iu'\cdot y'}du'\Big|dy'<\infty.
\]
The finiteness of the integral appearing in the second term on the right-hand side of (\ref{WUSeqlemma316}) can be similarly proved. Now, the assertion follows from (\ref{WUSeqlemma316}).

\end{proof}

The following two lemmas will be used to obtain a solvability result about the parabolic resolvent equation for $A$; however, they are interesting in their own right.
%%%%%%%%%%%%%%%%%%%%
%  Lemma 3.2
%%%%%%%%%%%%%%%%%%%%

\begin{lemma}\label{WUSlemma32}Let $(E,\CM, m)$ be a measure space and $f: \Rd\times E \to \BR$ be $\CB(\Rd)\otimes\CM$-measurable. Denote by $L^1(E,\CM,m)$ the space of all $\CM$-measurable functions on $E$ that are integrable with respect to the measure $m$. Suppose $f$ satisfies: \\
(i) For each $y \in E$, $f(\cdot,y) \in C^2_b(\Rd)$. Moreover, there exist $g_0, g_1,g_2 \in L^1(E,\CM, m)$ such that $|f(x,\cdot )| \le g_0$, $|\nabla_x f(x,\cdot )| \le g_1$ and $\sum_{i,j=1}^d|\partial^2_{x_ix_j} f(x,\cdot )| \le g_2$ for all $x \in \Rd$. \\
(ii) For each $x \in \Rd$, $f(x,\cdot) \in L^1(E,\CM, m)$. \\
Then
\[
A\Big(\int_E f(x,y)m(dy)\Big)=\int_E A_xf(x,y)m(dy).
\]
\end{lemma}
\begin{proof}
Let $h(x):=\int_E f(x,y)m(dy)$, $x\in\Rd$. By condition (i) and dominated convergence theorem, we have $h \in C^2_b(\Rd)$; in particular, $\nabla h(x)=\int_E \nabla_x f(x,y)m(dy)$. As a consequence, $Ah$ is well-defined.

For $x \in \Rd$, we have
\begin{align}
Ah(x)=& \gamma \cdot \nabla h(x)+ \int_{\mathbb{R}^{d}\setminus\{0\}}\big(h(x+z)-h(x)-z\cdot\nabla h(x)\big)\nu(dz) \notag\\
 =& \big(\gamma - \int_{\{|z|>1\}} z \nu(dz)\big)\cdot \nabla h(x)+ \int_{\{|z|>1\}}\big(h(x+z)-h(x)\big)\nu(dz)\notag \\
 & \qquad + \int_{\{0<|z|\le 1\}}\big(h(x+z)-h(x)-z\cdot\nabla h(x)\big)\nu(dz) \notag\\
\label{WUSeqlemma321}=& \big(\gamma - \int_{\{|z|>1\}} z \nu(dz)\big)\cdot \int_E \nabla_x f(x,y)m(dy)+ I_1+I_2,
\end{align}
where
\[
I_1:= \int_{\{|z|>1\}}\big(h(x+z)-h(x)\big)\nu(dz)
\]
and
\[
I_2:=\int_{\{0<|z|\le 1\}}\big(h(x+z)-h(x)-z\cdot\nabla h(x)\big)\nu(dz).
\]
Since $|f(x,\cdot )| \le g_0$ for all $x\in\Rd$, we can apply Fubini's theorem to obtain
\begin{align}
I_1=& \int_{\{|z|>1\}}\int_E \big(f(x+z,y)-f(x,y)\big)m(dy)\nu(dz) \notag \\
\label{WUSeqlemma322}=&\int_E\int_{\{|z|>1\}} \big(f(x+z,y)-f(x,y)\big)\nu(dz)m(dy).
\end{align}
Noting that $\sum_{i,j=1}^d|\partial^2_{x_ix_j} f(x,\cdot )| \le g_2$ for all $x \in \Rd$, we can find a constant $C>0$ such that
\[
|\nabla_x f(x+z,y)-\nabla_x f(x,y)| \le C g_2(y)|z|, \qquad x,z \in \Rd, \ y \in E.
\]
Therefore, for all $x,z \in \Rd, \ y \in E$,
\begin{align*}
|f(x+z,y)-f(x,y)-z \cdot \nabla_x f(x,y)| \le & \int_0^1 |\nabla_x f(x+r z,y)-\nabla_x f(x,y)||z|dr  \\
\le & C g_2(y)|z|^2.
\end{align*}
This allows us to use Fubini's theorem to obtain
\begin{align}
I_2=& \int_{\{0<|z|\le 1\}}\int_E \big (  f(x+z,y)-f(x,y)-z \cdot \nabla_x f(x,y)\big)m(dy)\nu(dz) \notag \\
\label{WUSeqlemma323}=&\int_E\int_{\{0<|z|\le 1\}} \big (  f(x+z,y)-f(x,y)-z \cdot \nabla_x f(x,y)\big)\nu(dz)m(dy).
\end{align}
Now, the assertion follows easily from (\ref{WUSeqlemma321}), (\ref{WUSeqlemma322}) and (\ref{WUSeqlemma323}).
\end{proof}

%%%%%%%%%%%%%%%%%%%%
%  Lemma 3.3
%%%%%%%%%%%%%%%%%%%%

\begin{lemma}\label{WUSlemma33}Let $g \in C^2_b(\Rd)$ and $h_n \in C_0^{\infty}(\Rd)$ be such that $0\le h_n \le 1$, $h_n(x)=1$ for $|x| \le n$ and $\ \sup_{n \in \BN} \|h_n\|_{C^2_b(\Rd)} <\infty$. Then $
A(gh_n) \to Ag$
boundedly and pointwise as $n \to \infty$.
\end{lemma}
\begin{proof}
Firstly, we note that $\sup_{n \in \BN} \|gh_n\|_{C^2_b(\Rd)}  <\infty$. Therefore,
\begin{align}
&|(gh_n)(x+z)-(gh_n)(x)-z\cdot\nabla (gh_n)(x)| \notag \\
\label{WUSeqlemma33}& \hspace*{3cm} \le  C (\mathbf{1}_{\{|z| > 1\}}+|z|\mathbf{1}_{\{|z| > 1\}}+|z|^2 \mathbf{1}_{\{|z| \le 1\}})
\end{align}
for all $x,z \in \Rd$ and $n \in \BN$, where $C>0$ is a constant. Thus there exists $C'>0$ such that $\|A(gh_n)\|\le C'$. On the other hand, it is easy to verify that
\begin{equation}\label{WUSeqlemma331}
\lim_{n\to\infty}(gh_n)(x)=g(x) \qquad \mbox{and} \qquad \lim_{n\to\infty} \nabla (gh_n)(x)=\nabla g(x)
\end{equation}
for all $x \in \Rd$. Since the right-hand side of (\ref{WUSeqlemma33}) is an integrable function with respect to the measure $\nu$, by (\ref{WUSeqlemma331}) and dominated convergence theorem, we obtain \[
\lim_{n\to \infty} A(gh_n)(x) = Ag(x), \quad x \in \Rd.
\]
This completes the proof.
\end{proof}

\begin{remark}The existence of a sequence of functions $(h_n)_{n\in \BN}$ being as in Lemma \ref{WUSlemma33}  is obvious. For example, we can take
\[
g(t):=\begin{cases}\exp \frac{1}{(t-1)(t-4)}, \quad & t \in (1,4), \\ 0,  & t \notin (1,4). \end{cases}
\]
Then let $F(t): =(\int_{-\infty}^{\infty}g(s)ds)^{-1} \int_t^{\infty}g(s)ds $ and $h_n(x):=F(|x|^2/n^2)$, $x \in \Rd$.
\end{remark}

%%%%%%%%%%%%%%%%%%%%
%  Proposition 3.4
%%%%%%%%%%%%%%%%%%%%
For $\lambda>0$, the time-space resolvent operator $R^{\lambda}$  of the stable process $S=(S_{t})_{t\geq 0}$ is defined by
\begin{equation}\label{WUSeqdefiGLam}
R^{\lambda}f(s,x):=\int_{0}^{\infty}\int_{\mathbb{R}^{d}}e^{-\lambda t}p_{t}(y-x)f(t+s,y)dydt,\quad  (s,x) \in \hs,
\end{equation}
where $f: \hs \to \BR$  is an arbitrary measurable function such that the integral on the right of (\ref{WUSeqdefiGLam}) is finite for all $(s,x) \in \hs$.

The following proposition is about the solvability of the parabolic resolvent equation for the generator $A$ of the stable process $S$. It plays a key role in obtaining a perturbative representation of the time-space resolvent of the solution to (\ref{WUSeqsect1}).
\begin{proposition}\label{WUSprop31}Suppose $\lambda>0$ and $g\in C^{1,2}_b(\hs)$. Let $f(s,x):=R^{\lambda}g(s,x)$, $(s,x) \in \hs$. Then $f$ belongs to $C^{1,2}_b(\hs)$ and solves the equation
\begin{equation}\label{WUSsolutionparaeq}
\lambda f-\partial_s f-Af=g \qquad \mathrm{on} \quad \hs,
\end{equation}
where $A$ is defined by (\ref{WUSgeneratorA}).
\end{proposition}
\begin{proof}
By the definition of $R^{\lambda}$, we have
\begin{align}
\label{WUSeqrep1}f(s,x)=&\int_{0}^{\infty}\int_{\mathbb{R}^{d}}e^{-\lambda t}p_t(y-x)g(t+s,y)dydt \\
\label{WUSeqrep2}=& \int_{0}^{\infty}\int_{\mathbb{R}^{d}}e^{-\lambda t}p_t(y)g(t+s,x+y)dydt
\end{align}
for $(s,x) \in \hs$. In the rest of this proof we will use either the representation (\ref{WUSeqrep1}) or (\ref{WUSeqrep2}), according to our needs.

Since $g\in C^{1,2}_b(\hs)$, the functions $|g|,|\partial_t g|, |\nabla g|$ and $|\partial^2_{ij} g |$, $i,j=1,\cdots,d$, are all bounded on $\hs$. It follows from (\ref{WUSeqrep1}) and dominated convergence theorem that $\partial_s f$ is bounded and continuous on $\hs$; moreover,
\[
\partial_s f(s,x)=\int_{0}^{\infty}\int_{\mathbb{R}^{d}}e^{-\lambda t}p_t(y-x)\partial_s ( g(t+s,y))dydt, \quad (s,x)\in \hs.
\]
Similarly, by (\ref{WUSeqrep2}), $\nabla f$ and $\partial^2_{ij} f $, $i,j=1,\cdots,d$, are also bounded and continuous. Thus $f\in C^{1,2}_b(\hs)$. Furthermore, it follows from (\ref{WUSeqrep2}) and Lemma \ref{WUSlemma32} that
\begin{equation}\label{WUSeqprop34new1}
A f(s,x)=\int_{0}^{\infty}\int_{\mathbb{R}^{d}}e^{-\lambda t}p_t(y)A_x(g(t+s,x+y))dydt.
\end{equation}

We are now in a position to define an approximating sequence $(f_{\epsilon})_{\epsilon>0}$ of $f$. In the following we will first derive an equation that $f_{\epsilon}$ fulfills, and then take the limit as $\epsilon \to 0$ to obtain (\ref{WUSsolutionparaeq}) for $f$.

Let $\epsilon >0$ and
\[
f_{\epsilon}(s,x):=\int_{\epsilon}^{\infty}\int_{\mathbb{R}^{d}}e^{-\lambda t}p_t(y)g(t+s,x+y)dydt.
\]
Then
\[
\partial_s f_{\epsilon}(s,x)=\int_{\epsilon}^{\infty}\int_{\mathbb{R}^{d}}e^{-\lambda t}p_t(y)\partial_s (g(t+s,x+y))dydt
\]
and
\begin{equation}\label{WUSeqprop341}
\lim_{\epsilon \to 0} \partial_s f_{\epsilon}(s,x)=\partial_s f (s,x), \quad (s,x)\in \hs .
\end{equation}
Noting (\ref{WUSsect315}), it follows from Lemma \ref{WUScoro31} that $\lim_{t\to\infty}p_t(x)=0$ for all $x \in \Rd$. By Fubini's theorem and integration by parts formula, we have
\begin{align}
\partial_s f_{\epsilon}(s,x)=& \int_{\mathbb{R}^{d}}\int_{\epsilon}^{\infty}e^{-\lambda t}p_t(y)\partial_s (g(t+s,x+y))dtdy \notag\\
=& \int_{\mathbb{R}^{d}}\int_{\epsilon}^{\infty}e^{-\lambda t}p_t(y)\partial_t (g(t+s,x+y))dtdy \notag\\
=&\int_{\mathbb{R}^{d}}e^{-\lambda t}p_t(y) g(t+s,x+y)\big|^{t=\infty}_{t=\epsilon}dy  \notag \\
& \quad -\int_{\mathbb{R}^{d}} \int_{\epsilon}^{\infty}g(t+s,x+y)\partial_t\big( e^{-\lambda t} p_t(y)\big)dtdy \notag\\
=&-\int_{\mathbb{R}^{d}}e^{-\lambda \epsilon}p_\epsilon(y) g(s+\epsilon,x+y)dy  \notag \\
& \quad +\int_{\mathbb{R}^{d}} \int_{\epsilon}^{\infty}e^{-\lambda t}g(t+s,x+y)(\lambda p_t(y)-\partial_t p_t(y))dtdy \notag \\
\label{WUSeqprop342}=:&I_{\epsilon}+J_{\epsilon}.
\end{align}
Obviously,
\begin{align}
\lim_{\epsilon \to 0}I_{\epsilon}=&-\lim_{\epsilon \to 0} \int_{\mathbb{R}^{d}}p_\epsilon(y) g(s+\epsilon,x+y)dy \notag \\
\label{WUSeqprop343}=& -\lim_{\epsilon \to 0}\bfE[g(s+\epsilon,x+S_{\epsilon})]=-g(s,x).
\end{align}
By Fubini's theorem and a change of variables,
\begin{align}
J_{\epsilon}=& \int_{\epsilon}^{\infty} \int_{\mathbb{R}^{d}} e^{-\lambda t}g(t+s,x+y)(\lambda p_t(y)-\partial_t p_t(y))dydt \notag \\
\label{WUSeqprop344}=& \int_{\epsilon}^{\infty} \int_{\mathbb{R}^{d}} e^{-\lambda t}(\lambda p_t(y-x)-\partial_t p_t(y-x) )g(t+s,y)dydt.
\end{align}

Just as in (\ref{WUSeqprop34new1}), we have
\[
Af_{\epsilon}(s,x)= \int_{\epsilon}^{\infty} \int_{\mathbb{R}^{d}} e^{-\lambda t} p_t(y) A_x(g(t+s,x+y))dydt,
\]
so
\begin{equation}\label{WUSeqprop345}
\lim_{\epsilon \to 0} Af_{\epsilon}(s,x) = \int_{0}^{\infty} \int_{\mathbb{R}^{d}} e^{-\lambda t} p_t(y) A_x(g(t+s,x+y))dydt=Af(s,x).
\end{equation}

Let $h_n$ be as in Lemma \ref{WUSlemma33}. By Lemma \ref{WUSlemma32}, Lemma \ref{WUSlemma33} and dominated convergence theorem, we have
\begin{align*}
Af_{\epsilon}(s,x)=& \int_{\epsilon}^{\infty} \int_{\mathbb{R}^{d}} e^{-\lambda t} p_t(y) A_x(g(t+s,x+y))dydt  \\
=& \lim_{n\to\infty}\int_{\epsilon}^{\infty} \int_{\mathbb{R}^{d}} e^{-\lambda t} p_t(y) A_x(h_n(x+y) g(t+s,x+y))dydt \\
=& \lim_{n\to\infty}A_x \Big(\int_{\epsilon}^{\infty} \int_{\mathbb{R}^{d}} e^{-\lambda t} p_t(y) h_n(x+y) g(t+s,x+y)dydt\Big) \\
=& \lim_{n\to\infty}A_x \Big(\int_{\epsilon}^{\infty} \int_{\mathbb{R}^{d}} e^{-\lambda t} p_t(y-x) h_n(y) g(t+s,y)dydt\Big) \\
=& \lim_{n\to\infty} \int_{\epsilon}^{\infty} \int_{\mathbb{R}^{d}} e^{-\lambda t} A_x(p_t(y-x)) h_n(y) g(t+s,y)dydt.
\end{align*}
Since $|h_n| \le 1$ and $g$ is also bounded, it follows from Lemma \ref{WUSlemma316} and  dominated convergence theorem that
\begin{equation} \label{WUSeqprop313}
Af_{\epsilon}(s,x)=\int_{\epsilon}^{\infty} \int_{\mathbb{R}^{d}} e^{-\lambda t} A_x(p_t(y-x)) g(t+s,y)dydt.
\end{equation}

Finally, we verify that $f$---as the limit of $f_{\epsilon}$---is a solution to the equation (\ref{WUSsolutionparaeq}). By (\ref{WUSeqprop342}) , (\ref{WUSeqprop344}) and (\ref{WUSeqprop313}), we have
\begin{align*}
&(\lambda f_{\epsilon}-\partial_sf_{\epsilon}-Af_{\epsilon} )(s,x)   \\
=& \lambda f_{\epsilon}(s,x)-I_{\epsilon}-J_{\epsilon}-\int_{\epsilon}^{\infty} \int_{\mathbb{R}^{d}} e^{-\lambda t} A_x(p_t(y-x)) g(t+s,y)dydt \\
=& \int_{\epsilon}^{\infty} \int_{\mathbb{R}^{d}} e^{-\lambda t}\big( \partial_tp_t(y-x)-\lambda p_t(y-x)-A_x(p_t(y-x))\big) g(t+s,y)dydt  \\
& \quad +\lambda f_{\epsilon}(s,x)-I_{\epsilon}.
\end{align*}
Since
\[
\partial_t p_t(y-x)=-\frac{1}{(2\pi)^d}\int_{\Rd}\psi(u)e^{-t\psi(u)}e^{-iu\cdot (y-x)}du, \quad x \in \Rd, t>0,
\]
it follows from (\ref{WUSeqremark3155}) that $\partial_t p_t(y-x)=A_x\big(p_t(y-x)\big)$,
which implies
\begin{align}
&(\lambda f_{\epsilon}-\partial_sf_{\epsilon}-Af_{\epsilon} )(s,x)  \notag \\
\label{WUSeqepsiloneq}=&\lambda f_{\epsilon}(s,x)-I_{\epsilon}-\lambda \int_{\epsilon}^{\infty}\int_{\mathbb{R}^{d}} e^{-\lambda t} p_t(y-x)g(t+s,y)dydt=-I_{\epsilon}.
\end{align}
Obviously, $f_{\epsilon}(s,x)$ converges to $f(s,x)$ as $\epsilon \to 0$. Letting $\epsilon \to 0$ in (\ref{WUSeqepsiloneq}), the equation (\ref{WUSsolutionparaeq}) follows from (\ref{WUSeqprop341}), (\ref{WUSeqprop343}) and (\ref{WUSeqprop345}).
\end{proof}

%%%%%%%%%%%%%%%%%%%%
%  Proposition 3.5
%%%%%%%%%%%%%%%%%%%%

\begin{proposition}\label{WUSprop32}Let $T>0$ and $f:\hs \to \BR$ be a measurable function such that $supp(f)\subset [0,T]\times \Rd$.\\
 (i) If $f\in L^q([0,T];L^p(\Rd))$ with $d/p+\alpha/q <\alpha$, then
 \[
 \|R^{\lambda}f\|\le N_{\lambda} \|f\|_{L^q([0,T];L^p(\Rd))},\]
where $N_{\lambda}>0$ is a constant depending on $\lambda$, $p$ and $q$. Moreover, $N_{\lambda} \downarrow 0$ as $\lambda \to \infty$.  \\
(ii) If $f\in L^q([0,T];L^p(\Rd))$ with $d/ p+\alpha/q <\alpha-1$, then
\[
\|\nabla R^{\lambda}f\| \le M_{\lambda}\|f\|_{L^q([0,T];L^p(\Rd))},\]
where $M_{\lambda}>0$ is a constant depending on $\lambda$, $p$ and $q$. Moreover, $M_{\lambda} \downarrow 0$ as $\lambda \to \infty$. \\
\end{proposition}
\begin{proof}
(i)  Since $supp(f)\subset [0,T]\times \Rd$, upon using H\"older's inequality twice, we get
\begin{align*}
|R^{\lambda}f(s,x)|=& \Big|\int_{0}^{\infty}e^{-\lambda t}\int_{\mathbb{R}^{d}}p_{t}(y-x)f(t+s,y)dydt\Big| \\
\le & \int_{0}^{\infty}e^{-\lambda t}\|p_{t}(\cdot-x)\|_{L^{p^*}(\Rd)}\|f(t+s,\cdot)\|_{L^p(\Rd)}dt \\
=& \int_{0}^{T}e^{-\lambda t}\|p_{t}\|_{L^{p^*}(\Rd)}\|f(t+s,\cdot)\|_{L^p(\Rd)}dt \\
\le &\Big(\int_{0}^{T}e^{-q^*\lambda t}\|p_{t}\|^{q^*}_{L^{p^*}(\Rd)}dt\Big)^{1/q^*}\|f\|_{L^q([0,T];L^p(\Rd))},
\end{align*}
where $p^*,q^*>0$ are such that $1/p^*+1/p=1$ and $1/q^*+1/q=1$. By the scaling property (\ref{WUSsect315}),
\begin{align*}
\|p_{t}\|_{L^{p^*}(\Rd)}=&\Big(\int_{\mathbb{R}^{d}}|p_{t}(y)|^{p^*}dy\Big)^{1/p^*} \\
=& \Big(\int_{\mathbb{R}^{d}}t^{-d(p^*-1)/\alpha}|p_{1}(y)|^{p^*}dy\Big)^{1/p^*} \\
=& t^{-d(p^*-1)/(\alpha p^*)}\|p_1\|_{L^{p^*}(\Rd)}.
\end{align*}
Thus the assertion holds with
\[
N_{\lambda}:= \Big(\int_{0}^{\infty}e^{-q^*\lambda t}t^{-dq^*(p^*-1)/(\alpha p^*)}dt\Big)^{1/q^*}\|p_1\|_{L^{p^*}(\Rd)},
\]
which is finite if $-dq^*(p^*-1)/(\alpha p^*)>-1$, or equivalently, $d/p+\alpha/q <\alpha$. By dominated convergence theorem, $\lim_{\lambda \to \infty } N_{\lambda}= 0$. \\
(ii) We first show that for fixed $t>0$,
\begin{equation}\label{WUSeqlemm341}
\nabla_x\Big( \int_{\mathbb{R}^{d}} p_{t}(y-x)f(t+s,y)dy\Big)=\int_{\mathbb{R}^{d}}\nabla_x( p_{t}(y-x))f(t+s,y)dy.
\end{equation}
To this end, set $f_n(t,y):=f(t,y)\mathbf{1}_{\{|y|\le n\}}(y)$, $(t,y) \in \hs$. By dominated convergence theorem and H\"older's inequlaity,
\begin{equation}\label{WUSeqlemm342}
 \int_{\mathbb{R}^{d}} p_{t}(y-x)f_n(t+s,y)dy \to \int_{\mathbb{R}^{d}} p_{t}(y-x)f(t+s,y)dy
\end{equation}
uniformly in $x \in \Rd$ as $n\to\infty$. Again by dominated convergence theorem,
\[
\nabla_x \Big(\int_{\mathbb{R}^{d}} p_{t}(y-x)f_n(t+s,y)dy\Big)=\int_{\mathbb{R}^{d}}\nabla_x (p_{t}(y-x))f_n(t+s,y)dy.
\]
Just as in (\ref{WUSeqlemm342}),
\[
 \int_{\mathbb{R}^{d}} \nabla_x (p_{t}(y-x))f_n(t+s,y)dy \to \int_{\mathbb{R}^{d}}\nabla_x (p_{t}(y-x))f(t+s,y)dy
\]
uniformly in $x \in \Rd$ as $n\to\infty$. Since $ \int_{\mathbb{R}^{d}} p_{t}(y-\cdot)f_n(t+s,y)dy \in C^1_b(\Rd)$ for each fixed $s,t \ge 0$ and $C^1_b(\Rd)$ is a Banach space, it follows that
\[\int_{\mathbb{R}^{d}} p_{t}(y-\cdot)f(t+s,y)dy \in C^1_b(\Rd)\]
and (\ref{WUSeqlemm341}) holds.

For $i=1,\cdots,d$, by (\ref{WUSsect315}), we get
\begin{align*}
\|\partial_i p_{t}\|_{L^{p^*}(\Rd)}=&\Big(\int_{\mathbb{R}^{d}}|\partial_{x_i}p_{t}(x)|^{p^*}dx\Big)^{1/p^*} \\
=& \Big(\int_{\mathbb{R}^{d}}t^{-p^*(d+1)/\alpha}|(\partial_ip_{1})(t^{-1/\alpha}x+(1-t^{1-1/\alpha})\gamma)|^{p^*}dx\Big)^{1/p^*} \\
=& \Big(\int_{\mathbb{R}^{d}}t^{-p^*(d+1)/\alpha+d/\alpha}|\partial_{x_i}p_{1}(x)|^{p^*}dx\Big)^{1/p^*} \\
=& t^{-(d+1)/\alpha+d/(\alpha p^*)} \|\partial_{i}p_1\|_{L^{p^*}(\Rd)}.
\end{align*}
As in (i), we can apply H\"older's inequality to obtain
\begin{align}
\Big|\int_{\mathbb{R}^{d}}\nabla_x( p_{t}(y-x))f(t+s,y)dy\Big| \le & \|\nabla_x(p_{t}(\cdot-x))\|_{L^{p^*}(\Rd)}\|f(t+s,\cdot)\|_{L^p(\Rd)} \notag \\
\le &  \|\nabla p_{t} \|_{L^{p^*}(\Rd)}\|f(t+s,\cdot)\|_{L^p(\Rd)} \notag \\
\label{WUSeqlemm343} \le & C t^{-(d+1)/\alpha+d/(\alpha p^*)}\|f(t+s,\cdot)\|_{L^p(\Rd)},
\end{align}
where $C>0$ is a constant. If $d/ p+\alpha/q <\alpha-1$, then
\[
-q^*(d+1)/\alpha+dq^*/(\alpha p^*)>-1.
\]
Since $f\in L^q([0,T];L^p(\Rd))$, by H\"older's inequality, we see that the right-hand side of  (\ref{WUSeqlemm343}) is an integrable  function (with the variable $t$) on $[0,T]$. Now, it follows from (\ref{WUSeqlemm341}),  (\ref{WUSeqlemm343}) and dominated convergence theorem that
\[
\nabla R^{\lambda}f(s,x)=\int_{0}^{\infty}\exp(-\lambda t)\int_{\mathbb{R}^{d}}\nabla_x( p_{t}(y-x))f(t+s,y)dydt.
\]
The rest of the proof is completely similar to that of  (i), and we omit the details. We can take
\[
M_{\lambda}:= C \Big(\int_{0}^{\infty}e^{-q^*\lambda t}t^{-q^*(d+1)/\alpha+dq^*/(\alpha p^*)}dt\Big)^{1/q^*}.
\]
\end{proof}

Similarly to Proposition \ref{WUSprop32} (ii), we have the following  estimate for $R^{\lambda}$. Its proof is very simple and is thus omitted.

%%%%%%%%%%%%%%%%%%%%
%  Lemma 3.6
%%%%%%%%%%%%%%%%%%%%

\begin{lemma}\label{WUSlemma34}
For each $\lambda >0$, there exists a constant $L_{\lambda}>0$ such that
\[
    \|\nabla R^{\lambda}g\|\le L_{\lambda}  \|g\|_{L^{\infty}([0,\infty);L^{\infty}(\mathbb{R}^{d}))}\]
for all $g \in L^{\infty}([0,\infty);L^{\infty}(\mathbb{R}^{d}))$, where $L_{\lambda} \downarrow 0$ as $\lambda \to \infty$.
\end{lemma}

%%%%%%%%   Section 4: Existence and uniqueness of weak solution: local case

\section{Existence and Uniqueness of Weak Solutions: Local Case}

%%%%%%%%%%%%%%%%%%%%
%  Assumption 4.1
%%%%%%%%%%%%%%%%%%%%

In this section,  we confine ourselves to the local case and thus assume in addition to Assumption \ref{WUSass21} the following:

\begin{assumption}\label{WUSassumption41} The drift $b:\hs \to \Rd$ is such that $supp \big(b\big ) \subset [0,T]\times \Rd$ and  $b \in L^{\infty}([0,T];L^{\infty}(\mathbb{R}^{d}))+ L^{q}([0,T];L^{p}(\mathbb{R}^{d}))$ for some $T,p,q>0$ with $d/ p+\alpha/q <\alpha-1$.
 \end{assumption}

We first consider smooth approximations of the singular drift $b$. According to Assumption \ref{WUSassumption41}, we can assume $b=b_1+b_2$ with $supp \big(b_i\big ) \subset [0,T]\times \Rd$ for $i=1,2$, and
\begin{equation}\label{WUSeqsect4-1}
\|b_1\|_{L^{\infty}([0,T];L^{\infty}(\mathbb{R}^{d}))} \le  M, \quad  \|b_2 \|_{ L^{q}([0,T];L^{p}(\mathbb{R}^{d}))}<\infty,
\end{equation}
where $M>0$ is a constant and $d/ p+\alpha/q <\alpha-1$.  Let
\[
 \hat{b}_{1,n}:=b_1 \mathbf{1}_{\{|b_1| \le  n\}}, \quad \hat{b}_{2,n}:=b_2 \mathbf{1}_{\{|b_2| \le  n\}}.
 \]
Next,  for $(t,y) \in \hs$, define
 \begin{equation}\label{WUSeqsect40}
 b^{(n)}_{1}(t,y):=(\hat{b}_{1,n}(t,\cdot)\ast \varphi_n)(y),  \quad b^{(n)}_{2}(t,y):=(\hat{b}_{2,n}(t,\cdot)\ast \varphi_n)(y),
\end{equation}
where $(\varphi_n)_{n \in \BN}$ is a mollifying sequence on $\Rd$. Then $b^{(n)}_{1}$ and $b^{(n)}_{2}$ are both bounded and globally Lipschitz in the space variable. Finally, let
\[
 b^{(n)}:=b^{(n)}_{1}+b^{(n)}_{2}.
\]
Obviously,
\begin{equation}\label{WUSeqsect41}
 supp \big(b^{(n)}\big ) \subset [0,T]\times \Rd \quad \mbox{and} \quad \|b^{(n)}\| \le 2n.
\end{equation}
%%%%%%%%%%%%%%%%%%%%
%  Remark 4.15
%%%%%%%%%%%%%%%%%%%%

Since $\|b_{1}\|_{L^{\infty}([0,T];L^{\infty}(\mathbb{R}^{d}))}\le M$, it is easy to see that
\begin{equation}\label{WUSeqb1nlinftynorm}
\|b^{(n)}_{1}\|_{L^{\infty}([0,T];L^{\infty}(\mathbb{R}^{d}))} \le M .
\end{equation}

%%%%%%%%%%%%%%%%%%%%
%  Remark 4.2
%%%%%%%%%%%%%%%%%%%%
\begin{remark}\label{WUSremark42}For each fixed $t\ge0$, it follows from Young's inequality that
 \begin{equation}\label{WUSeqremark411}
 \|b^{(n)}_{2}(t,\cdot)\|_{L^{p}(\mathbb{R}^{d})}\le \|\hat{b}_{2,n}(t,\cdot)\|_{L^{p}(\mathbb{R}^{d})}\le \|b_{2}(t,\cdot)\|_{L^{p}(\mathbb{R}^{d})}.
 \end{equation}
 Therefore,
 \begin{equation}\label{WUSeqb2nlpqnorm}
 \|b^{(n)}_{2}\|_{L^{q}([0,T];L^{p}(\mathbb{R}^{d}))}\le \|b_{2}\|_{L^{q}([0,T];L^{p}(\mathbb{R}^{d}))}.
 \end{equation}
 If $t\ge 0$ is such that $\|b_{2}(t,\cdot)\|_{L^{p}(\mathbb{R}^{d})}<\infty$, then
 \begin{align}
 &\lim_{n \to \infty}\|b^{(n)}_{2}(t,\cdot)-b_{2}(t,\cdot)\|_{L^{p}(\mathbb{R}^{d})}\notag \\
 =& \lim_{n \to \infty}\|\hat{b}_{2,n}(t,\cdot)\ast \varphi_n-b_2(t,\cdot)\ast \varphi_n+b_2(t,\cdot)\ast \varphi_n- b_{2}(t,\cdot)\|_{L^{p}(\mathbb{R}^{d})} \notag \\
 \le & \limsup_{n \to \infty}\|\hat{b}_{2,n}(t,\cdot)-b_2(t,\cdot)\|_{L^{p}(\mathbb{R}^{d})}+ \limsup_{n \to \infty}\|b_2(t,\cdot)\ast \varphi_n-b_2(t,\cdot)\|_{L^{p}(\mathbb{R}^{d})}\notag \\
 \label{WUSeqremark412}=&0.
 \end{align}
 It follows from (\ref{WUSeqremark411}), (\ref{WUSeqremark412}) and dominated convergence theorem that
 \begin{equation}\label{WUSeqremark413}
 \lim_{n \to \infty}\|b^{(n)}_{2}-b_{2}\|_{L^{q}([0,T];L^{p}(\mathbb{R}^{d}))}=0.
 \end{equation}
 \end{remark}

Now, consider an $\alpha$-stable ($1<\alpha<2$) process $S=(S_t)_{t\ge0}$  defined on some probability space $(\Omega, \CA, \bfP)$. As before, we assume that $S$ fulfills Assumption \ref{WUSass21}, that is, $S$ is non-degenerate. Recall that $R^{\lambda}$ is the time-space resolvent of $S$ and is defined in (\ref{WUSeqdefiGLam}).

Define an operator $BR^{\lambda}$ as follows. Given a function $f:\hs \to \BR$  for which $\nabla( R^{\lambda}f)$ is everywhere defined, define $BR^{\lambda}f :\hs \to \BR$  by
\begin{equation}\label{defiofBR}
    BR^{\lambda}f(t,y):=b(t,y) \cdot \nabla R^{\lambda}f(t,y), \quad (t,y) \in \hs.
\end{equation}
For example, $BR^{\lambda}f$ is well-defined if $f \in L^{\infty}(\BR_+;L^{\infty}(\mathbb{R}^{d}))$. Similarly, define $B_nR^{\lambda}f $ as
\begin{equation}\label{defiofBnR}
B_nR^{\lambda}f(t,y):=b^{(n)}(t,y) \cdot \nabla R^{\lambda}f(t,y), \quad (t,y) \in \hs,
\end{equation}
provided that $\nabla R^{\lambda}f$ exists everywhere.

Let $M_{\lambda}$ and $L_{\lambda}$ be as in Proposition \ref{WUSprop32} and Lemma \ref{WUSlemma34}, respectively. Since $M_{\lambda} \downarrow 0$ and $L_{\lambda} \downarrow 0$ as $\lambda \to \infty$, we can  find $\lambda_0>0$  such that
\begin{equation}\label{WUSeqsect42}
L_{\lambda_0}M+M_{\lambda_0}\|b_2\|_{L^q([0,T];L^p(\Rd))}< 1,
\end{equation}
where $M>0$ is the constant appearing in (\ref{WUSeqsect4-1}). If $\lambda> \lambda_{0}$, then $L_{\lambda} \le L_{\lambda_0}$ and $M_{\lambda} \le M_{\lambda_0}$. In view of (\ref{WUSeqsect42}), we have
\begin{equation}\label{WUSeqlemma437}
\kappa_{\lambda}:=L_{\lambda}M+M_{\lambda}\|b_2\|_{L^q([0,T];L^p(\Rd))}<1
\end{equation}
for any $ \lambda> \lambda_{0}$.

Note that $b^{(n)}$ is bounded and globally Lipschitz in the space variable. We now consider an $\alpha$-stable process with drift $b^{(n)}$.
%%%%%%%%%%%%%%%%%%%%
%  Lemma 4.3
%%%%%%%%%%%%%%%%%%%%
\begin{lemma}\label{newlemmasn}Let $\lambda_0>0$ and $\kappa_{\lambda}$ be as in (\ref{WUSeqsect42}) and (\ref{WUSeqlemma437}), respectively. Suppose $(s,x) \in \hs$. Let $X=(X_t)_{t\ge0}$ be the unique strong solution to the SDE
\begin{equation}\label{WUSeqsdewithdriftbn}
\begin{cases}dX_{t}=dS_{t}+b^{(n)}(s+t, X_{t})dt, & \ t\ge 0,\\
   \   X_{0}=x.
    \end{cases}
\end{equation}
Then for any $\lambda>\lambda_0$ and $g \in \mathcal{B}_{b}(\hs)$, we have
\begin{equation}\label{eqressn}
    \mathbf{E}\Big[\int_{0}^{\infty}e^{-\lambda t}f(t+s,X_{t})dt\Big]=\sum_{k=0}^{\infty}R^{\lambda}(B_{n}R^{\lambda})^{k}g(s,x),
\end{equation}
where $B_nR^{\lambda}$ is defined by (\ref{defiofBnR}). Moreover,  for each $k \in \BN$,
\begin{equation}\label{WUSeqlemma430}
\|R^{\lambda}(B_{n}R^{\lambda})^{k}g\|
 \leq  L_{\lambda}\Vert g\Vert (\kappa_{\lambda})^{k-1}\big(M\lambda^{-1}+N_{\lambda}\|b_2\|_{L^q([0,T];L^p(\Rd))}\big),
\end{equation}
which means that the series on the right-hand side of (\ref{eqressn}) converges and its convergence rate is independent of $(s,x)$ and $n$.
\end{lemma}

\begin{proof}
For the existence and uniqueness of strong solutions to the SDE (\ref{WUSeqsdewithdriftbn}), the reader is referred to \cite[Theorem 9.1]{MR1011252} and \cite[Theorem 117]{MR2160585}.

For $\lambda>0$ and $f  \in \mathcal{B}_{b}(\hs)$, define
\[
V^{\lambda}_{n}f:=\mathbf{E}\Big[\int_{0}^{\infty}e^{-\lambda t}f(t+s,X_{t})dt\Big]. \]

Applying It\^{o}'s formula for $f\in C^{1,2}_b(\hs)$, we obtain
\begin{align*}& f(t+s,X_{t})-f(s,X_{0}) \\
=& ``Martingale"+\int_{0}^{t}(\frac{\partial f}{\partial u}+L^{(n)}_{u}f)(u+s,X_{u})du,
\end{align*}
where $L^{(n)}_{u}:=A+b^{(n)}(u,\cdot) \cdot \nabla$ for $u\ge0$. Taking expectations of both sides of the above equality gives
\begin{equation}\label{thmunieq0}
    \mathbf{E}[f(t+s,X_{t})]-f(s,x)= \mathbf{E}\Big[\int_{0}^{t}(\frac{\partial f}{\partial u}+L^{(n)}_{u}f)(u+s,X_{u})du\Big].
\end{equation}
Note that
\[
\mathbf{E}\Big[\int_{0}^{\infty} e^{-\lambda t}|b^{(n)}(t+s,X_{t})|dt\Big] <\infty.
\]
Multiplying both sides of (\ref{thmunieq0}) by $e^{-\lambda t}$, integrating with respect to $t$ from $0$ to $\infty$ and then applying Fubini's theorem, we get
\begin{align*}&  \mathbf{E}\Big[\int_{0}^{\infty} e^{-\lambda t}f(t+s,X_{t})dt\Big]\\
=&\frac{1}{\lambda}f(s,x)+  \mathbf{E}\Big[\int_{0}^{\infty} e^{-\lambda t}\int_{0}^{t}\big(\frac{\partial f}{\partial u}+L^{(n)}_{u}f\big)(u+s,X_{u})dudt\Big]\\
=& \frac{1}{\lambda}f(s,x)+\frac{1}{\lambda} \mathbf{E}\Big[\int_{0}^{\infty} e^{-\lambda u}\big(\frac{\partial f}{\partial u}+L^{(n)}_{u}f\big)(u+s,X_{u})du\Big].
\end{align*}
Therefore, for $f \in C^{1,2}_b(\hs)$,
\begin{equation}\label{thmunieq101}
  \lambda V_{n}^{\lambda}f=f(s,x)+V_{n}^{\lambda}\Big(\frac{\partial f}{\partial t}+L^{(n)}_{t}f\Big).
\end{equation}
Given $g \in C^{1,2}_{b}(\hs)$, it follows from Proposition \ref{WUSprop31} that $f:=R^{\lambda}g \in C^{1,2}_b(\hs)$ and $(\lambda -A-\frac{\partial }{\partial t})f=g$.  Substituting this $f$ in (\ref{thmunieq101}), we obtain
\begin{equation}\label{neweqvn0}
    V_{n}^{\lambda}g=R^{\lambda}g(s,x)+V_{n}^{\lambda}(B_nR^{\lambda}g)
\end{equation}
for  $g \in C^{1,2}_{b}(\hs)$. After a standard approximation procedure, the equality (\ref{neweqvn0}) holds for any bounded continuous function $g$ on $\hs$. For any open subset $O $ of $ \hs$, we can find $f_k \in C_b(\hs)$, $k \in \BN$, such that $ 0 \le f_k \uparrow \bfi_O$ as $k\to \infty$. It is easy to see that $R^{\lambda}f_k$ and $\nabla R^{\lambda}f_k$ converge boundedly and pointwise to $R^{\lambda}\bfi_O$ and $\nabla R^{\lambda}\bfi_O$, respectively. By dominated convergence theorem,  (\ref{neweqvn0}) holds for $g=\bfi_O$. Then we can use a  monotone class argument (see, for example, \cite[p.~4]{MR648601}) to extend  (\ref{neweqvn0}) to every $g \in \mathcal{B}_{b}(\hs)$.

Therefore, we have shown that
\begin{equation}\label{WUSeqlemma43vnl}
 V_{n}^{\lambda}f-R^{\lambda}f(s,x)=V_{n}^{\lambda}(B_{n}R^{\lambda}f), \quad f \in \mathcal{B}_{b}(\hs). \end{equation}
For any bounded measurable function $g$ on $\hs$, taking $f=B_{n}R^{\lambda}g$ in (\ref{WUSeqlemma43vnl}), we get
\[
     V_{n}^{\lambda}(B_{n}R^{\lambda}g)-R^{\lambda}B_{n}R^{\lambda}g(s,x)=V_{n}^{\lambda}(B_{n}R^{\lambda})^2 g \]
and thus
\begin{align*}
    V_{n}^{\lambda}g=& R^{\lambda}g(s,x)+V_{n}^{\lambda}(B_{n}R^{\lambda}g) \\
    =& R^{\lambda}g(s,x)+R^{\lambda}B_{n}R^{\lambda}g(s,x)+V_{n}^{\lambda}(B_{n}R^{\lambda})^{2}g.
\end{align*}
Similarly, after $i$ steps, we obtain
\begin{equation}\label{eqsnforn}
 V_{n}^{\lambda}g=\sum_{k=0}^{i}R^{\lambda}(B_{n}R^{\lambda})^{k}g(s,x)+V_{n}^{\lambda}(B_{n}R^{\lambda})^{i+1}g, \quad g \in \mathcal{B}_{b}(\hs).
\end{equation}
In order to show that the first term on the right-hand side of (\ref{eqsnforn}) converges  as $i \to \infty$, we first need to prove the following claim.
\begin{claim}\label{WUSclaim1}
Suppose that $g: \hs \to \BR$ is such that $\Vert \nabla R^{\lambda} g\Vert<\infty$. Then for each $k\in \mathbb{Z}_{+}$,
\begin{equation}\label{WUSeqlemma436}
\Vert\nabla R^{\lambda}(B_{n}R^{\lambda})^{k}g\Vert \leq \Vert \nabla R^{\lambda} g\Vert \big(\kappa_{\lambda}\big)^{k}
\end{equation}
and
\begin{equation}\label{WUSeqlemma438}
|(B_{n}R^{\lambda})^{k+1}g|\leq \Vert \nabla R^{\lambda} g\Vert \big(\kappa_{\lambda}\big)^{k}(|b^{(n)}_{1}|+|b^{(n)}_{2}|).
\end{equation}

\end{claim}

We prove Claim \ref{WUSclaim1} by induction. If $k=0$, then  (\ref{WUSeqlemma436}) is trivial and
\[
|B_{n}R^{\lambda}g|\leq  \Vert \nabla R^{\lambda} g\Vert |b^{(n)}|
\leq \Vert \nabla R^{\lambda} g\Vert (|b^{(n)}_{1}|+|b^{(n)}_{2}|).
\]
Suppose now that the above claim is true for $k.$ Note that (\ref{WUSeqsect41}), (\ref{WUSeqb1nlinftynorm}) and (\ref{WUSeqb2nlpqnorm}) hold. By Proposition \ref{WUSprop32} and Lemma \ref{WUSlemma34}, we get
\begin{align*}
 \Vert\nabla R^{\lambda}(B_{n}R^{\lambda})^{k+1}g\Vert =& \Vert\nabla R^{\lambda}(b^{(n)}\cdot \nabla R^{\lambda}(B_{n}R^{\lambda})^{k}g)\Vert  \\
 \le & \Vert\nabla R^{\lambda}(b^{(n)}_{1}\cdot \nabla R^{\lambda}(B_{n}R^{\lambda})^{k}g)\Vert+ \Vert\nabla R^{\lambda}(b^{(n)}_{2}\cdot \nabla R^{\lambda}(B_{n}R^{\lambda})^{k}g)\Vert \\
 \le & L_{\lambda} \Vert b^{(n)}_{1}\cdot \nabla R^{\lambda}(B_{n}R^{\lambda})^{k}g \Vert_{L^{\infty}([0,\infty);L^{\infty}(\mathbb{R}^{d}))} \\
 & \quad + M_{\lambda}\Vert b^{(n)}_{2}\cdot \nabla R^{\lambda}(B_{n}R^{\lambda})^{k}g\Vert_{L^q([0,T];L^p(\Rd))} \\
\leq & \Vert \nabla R^{\lambda} g\Vert \big(\kappa_{\lambda}\big)^{k}\big( L_{\lambda}M+M_{\lambda}\|b^{(n)}_{2}\|_{L^q([0,T];L^p(\Rd))}\big) \\
\leq & \Vert \nabla R^{\lambda} g\Vert \big(\kappa_{\lambda}\big)^{k}\big( L_{\lambda}M+M_{\lambda}\|b_2\|_{L^q([0,T];L^p(\Rd))}\big) \\
=& \Vert \nabla R^{\lambda} g\Vert \big(\kappa_{\lambda}\big)^{k+1}
\end{align*}
and
\[
|(B_{n}R^{\lambda})^{k+2}g|\leq  \Vert \nabla R^{\lambda} g\Vert \big(\kappa_{\lambda}\big)^{k+1}(|b^{(n)}_{1}|+|b^{(n)}_{2}|).
\]
Thus the claim is also true for $k+1$. Hence Claim \ref{WUSclaim1} is true for any $k \in \BZ_+$.

Note that $\|R^{\lambda}f\| \le \lambda^{-1}\|f\|_{L^{\infty}(\BR_+;L^{\infty}(\Rd))}$ for all $f \in L^{\infty}(\BR_+;L^{\infty}(\Rd))$. By  (\ref{WUSeqlemma438}) and Proposition \ref{WUSprop32}, we obtain
\begin{align}
\|R^{\lambda}(B_{n}R^{\lambda})^{k}g\| \le & \Vert \nabla R^{\lambda} g\Vert \big(\kappa_{\lambda}\big)^{k-1} R^{\lambda}(|b^{(n)}_{1}|+|b^{(n)}_{2}|) \notag \\
\label{WUSeqlemma439} \leq &  \Vert \nabla R^{\lambda} g\Vert \big(\kappa_{\lambda}\big)^{k-1}\big(M\lambda^{-1}+N_{\lambda}\|b_2\|_{L^q([0,T];L^p(\Rd))}\big).
\end{align}
If $g \in \mathcal{B}_{b}(\hs)$, then $\Vert\nabla R^{\lambda}g\Vert \leq L_{\lambda}\Vert g\Vert $ by Lemma \ref{WUSlemma34}. So the inequality  (\ref{WUSeqlemma430}) is proved. By (\ref{WUSeqlemma437}) and (\ref{WUSeqlemma430}), we see that the series $ \sum_{k=0}^{\infty}R^{\lambda}(B_{n}R^{\lambda})^{k}g$ converges uniformly on $\mathbb{R}_{+}\times \mathbb{R}^{d}$ for any $\lambda>\lambda_{0}$ and  $g\in \mathcal{B}_{b}(\mathbb{R}_{+}\times \mathbb{R}^{d})$.

Finally, we show that the second term on the right-hand side of (\ref{eqsnforn}) converges to 0 as $i \to \infty$. Note that $|b^{(n)}_{1}|$ and $|b^{(n)}_{2}|$ are both bounded by $n$. According to Claim \ref{WUSclaim1},  we have, for any $\lambda>\lambda_{0}$ and $g \in \mathcal{B}_{b}(\hs)$,
\begin{align*}
 |V_{n}^{\lambda}(B_{n}R^{\lambda})^{i+1}g|
\leq& \Vert \nabla R^{\lambda} g\Vert \big(\kappa_{\lambda}\big)^{i}V_{n}^{\lambda}(|b^{(n)}_{1}|+|b^{(n)}_{2}|) \\
\leq & 2n \lambda^{-1}L_{\lambda}\Vert g\Vert \big(\kappa_{\lambda}\big)^{i},
\end{align*}
which converges to $0$ as $ i\rightarrow\infty$. Now, the equality (\ref{eqressn}) follows from (\ref{WUSeqlemma430}) and (\ref{eqsnforn}). This completes the proof.
\end{proof}

In view of (\ref{WUSeqlemma430}),  we can define an operator $G^{\lambda}_{n}$ on $\mathcal{B}_{b}(\hs)$ as
\begin{equation}\label{WUSeqdefiforGlambdan1}
    G_{n}^{\lambda}g=\sum_{k=0}^{\infty}R^{\lambda}(B_{n}R^{\lambda})^{k}g, \quad g  \in \mathcal{B}_{b}(\hs),
\end{equation}
provided that $\lambda> \lambda_0$. In the next lemma we study the limiting behavior of $G_{n}^{\lambda}$ as $n\to\infty$.
%%%%%%%%%%%%%%%%%%%%
%  Lemma 4.4
%%%%%%%%%%%%%%%%%%%%
\begin{lemma}\label{WUSlemma44}
Let $\lambda>\lambda_{0}$. For $g\in \mathcal{B}_{b}(\mathbb{R}_{+}\times \mathbb{R}^{d})$, define
\begin{equation}\label{WUSeqlemma440}
G^{\lambda}g:= \sum_{k=0}^{\infty}R^{\lambda}(BR^{\lambda})^{k}g.
\end{equation}
(i) Then the series on the right-hand side of (\ref{WUSeqlemma440}) converges uniformly on $\mathbb{R}_{+}\times \mathbb{R}^{d}$ for any $g\in \mathcal{B}_{b}(\mathbb{R}_{+}\times \mathbb{R}^{d})$. \\

\noindent (ii) For each $g\in \mathcal{B}_{b}(\mathbb{R}_{+}\times\mathbb{R}^{d}),\ G_{n}^{\lambda}g$ converges locally uniformly  to $G^{\lambda}g$ as $ n\rightarrow\infty$, that is, for any compact $K\subset \hs$,
\[
\lim_{n\to\infty}\sup_{(s,x)\in K}|\ G_{n}^{\lambda}g(s,x)-\ G^{\lambda}g(s,x)|=0.
\]
\end{lemma}
\begin{proof}
(i) Let $g\in \mathcal{B}_{b}(\mathbb{R}_{+}\times \mathbb{R}^{d})$. With the same argument that we used to establish (\ref{WUSeqlemma436}) and (\ref{WUSeqlemma439}), we conclude that, for each $k \in \BN$,
\begin{equation}\label{WUSeqlemma441}
\Vert\nabla R^{\lambda}(BR^{\lambda})^{k}g\Vert \leq \Vert\nabla R^{\lambda}g\Vert \big(\kappa_{\lambda}\big)^{k}
\end{equation}
and
\begin{equation}\label{WUSeqlemma442}
 \Vert R^{\lambda}(BR^{\lambda})^{k}g\Vert
 \leq  \Vert \nabla R^{\lambda} g\Vert \big(\kappa_{\lambda}\big)^{k-1}\big(M\lambda^{-1}+N_{\lambda}\|b_2\|_{L^q([0,T];L^p(\Rd))}\big).
\end{equation}
As before, we have $\Vert\nabla R^{\lambda}g\Vert \leq L_{\lambda}\Vert g\Vert $ by Lemma \ref{WUSlemma34}. Noting (\ref{WUSeqlemma437}), we see that the series $ \sum_{k=0}^{\infty}R^{\lambda}(BR^{\lambda})^{k}g$ converges uniformly on $\mathbb{R}_{+}\times \mathbb{R}^{d}$ for any $\lambda>\lambda_{0}$.

(ii) Suppose $g\in \mathcal{B}_{b}(\mathbb{R}_{+}\times \mathbb{R}^{d})$. By (\ref{WUSeqlemma437}) and the estimates  (\ref{WUSeqlemma430}) and (\ref{WUSeqlemma442}), we only need to show the following claim.
\begin{claim}\label{WUSclaim2}
For any fixed $k\in \BZ_+$ and  compact $K\subset \hs$,
\begin{equation}\label{WUSeqlemma443}
\lim_{n\rightarrow\infty}\sup_{(s,x)\in K}|R^{\lambda}(B_{n}R^{\lambda})^{k}g(s,x)-R^{\lambda}(BR^{\lambda})^{k}g(s,x)|=0
\end{equation}
and
\begin{equation}\label{WUSeqlemma444}
\lim_{n\rightarrow\infty}\sup_{(s,x)\in K}|\nabla R^{\lambda}(B_{n}R^{\lambda})^{k}g(s,x) -\nabla R^{\lambda}(BR^{\lambda})^{k}g(s,x)|=0.
\end{equation}
\end{claim}

We prove Claim \ref{WUSclaim2} by induction. If $k=0$, then (\ref{WUSeqlemma443}) and (\ref{WUSeqlemma444}) are trivially true. Suppose that the above claim is true for $k$. For $m>0$ let
\begin{equation}\label{WUSdefiofAm}
A_m:=\{x \in \Rd: |x|\le m\}.
\end{equation}
Define $h_m: \hs \to \BR$ by
\begin{equation}\label{WUSdefiofhm}
h_m(t,y):=\mathbf{1}_{A_{m}}(y),  \quad (t,y) \in \hs.
\end{equation}
Now set
\[
C_{n,m}:=\sup_{(s,x)\in [0,T]\times A_m}|\nabla R^{\lambda}(B_{n}R^{\lambda})^{k}g(s,x) -\nabla R^{\lambda}(BR^{\lambda})^{k}g(s,x)|.
\]
By induction hypothesis,  $\lim_{n\to \infty}C_{n,m}=0$ for any $m>0$.

Since the support of $b^{(n)}$ and $b$ are both contained in $[0,T]\times \Rd$, it follows that
\[
\nabla R^{\lambda}(B_{n}R^{\lambda})^{k}g(s,x)=\nabla R^{\lambda}(BR^{\lambda})^{k}g(s,x)=0, \qquad \forall s>T, x \in \Rd.
\]
By (\ref{WUSeqlemma436}) and (\ref{WUSeqlemma441}), we have
\begin{align*}
&|(B_{n}R^{\lambda})^{k+1}g-(BR^{\lambda})^{k+1}g|\\
\le &|(B_{n}R^{\lambda})^{k+1}g-B_{n}R^{\lambda}(BR^{\lambda})^{k}g|+|B_{n}R^{\lambda}(BR^{\lambda})^{k}g-(BR^{\lambda})^{k+1}g|\\
\leq &|\nabla R^{\lambda}(B_{n}R^{\lambda})^{k}g-\nabla R^{\lambda}(BR^{\lambda})^{k}g||b^{(n)}|+\Vert\nabla R^{\lambda}(BR^{\lambda})^{k}g\Vert |b^{(n)}-b| \\
= &|\nabla R^{\lambda}(B_{n}R^{\lambda})^{k}g-\nabla R^{\lambda}(BR^{\lambda})^{k}g||b^{(n)}|(h_m+(1-h_m))\\
& \quad +\Vert\nabla R^{\lambda}(BR^{\lambda})^{k}g\Vert |b^{(n)}-b| \\
 \le & C_{n,m}|b^{(n)}|+ 2C|b^{(n)}|(1-h_m)+ C|b^{(n)}-b|,\\
\end{align*}
where
\[
C:=\Vert\nabla R^{\lambda}g\Vert  \big(\kappa_{\lambda}\big)^{k} \le L_{\lambda}\Vert g\Vert \big(\kappa_{\lambda}\big)^{k}<\infty
\]
is a constant. Therefore,
\begin{align}
& |R^{\lambda}(B_{n}R^{\lambda})^{k+1}g-R^{\lambda}(BR^{\lambda})^{k+1}g| \notag \\
\le & C_{n,m} R^{\lambda}(|b^{(n)}|)+ 2C R^{\lambda}(|b^{(n)}|(1-h_m)) +C   R^{\lambda}(|b^{(n)}-b|) \notag \\
\label{WUSeqlemma4441}\le & C_{n,m} \|R^{\lambda}(|b^{(n)}|)\|+ 2C R^{\lambda}(|b^{(n)}|(1-h_m)) +C   R^{\lambda}(|b^{(n)}-b|).
\end{align}
By (\ref{WUSeqb1nlinftynorm}), (\ref{WUSeqb2nlpqnorm}) and Proposition \ref{WUSprop32}, we have
\begin{align}
\sup_{n\in\BN}\|R^{\lambda}(|b^{(n)}|)\| \le & \sup_{n\in\BN}\|R^{\lambda}(|b^{(n)}_{1}|+|b^{(n)}_{2}|)\| \notag \\
\le &  \sup_{n\in\BN} \big(\lambda^{-1}M+N_{\lambda}\|b^{(n)}_{2}\|_{L^q([0,T];L^p(\Rd))}\big) \notag \\
\label{WUSeqsuprlambdabn}\le &  \lambda^{-1}M+N_{\lambda}\|b_{2}\|_{L^q([0,T];L^p(\Rd))}< \infty,
\end{align}
which implies
\begin{equation}\label{WUSeqlemma4442}
\lim_{n\to\infty} C_{n,m} \|R^{\lambda}(|b^{(n)}|)\|=0, \quad \forall  m>0.
\end{equation}
Similarly,
\begin{align}
R^{\lambda}(|b^{(n)}-b|)
\le & R^{\lambda}(|b^{(n)}_{1}-b_{1}|+|b^{(n)}_{2}-b_{2}|)\notag\\
\label{WUSeqlemma445}\leq & R^{\lambda}(|b^{(n)}_{1}-b_{1}|)+N_{\lambda}\Vert b^{(n)}_{2}-b_{2}\Vert_{L^q([0,T];L^p(\Rd))}.
\end{align}
For $(s,x) \in \hs $,
\begin{align}
&R^{\lambda}(|b^{(n)}_{1}-b_{1}|)(s,x)\notag \\
=& R^{\lambda}(|b^{(n)}_{1}-b_{1}|h_m)(s,x) + R^{\lambda}(|b^{(n)}_{1}-b_{1}|(1-h_m))(s,x) \notag \\
\le & N_{\lambda}\|(|b^{(n)}_{1}-b_{1}|)h_m\|_{L^q([0,T];L^p(\Rd))}+2M\int_{0}^{\infty}\int_{\{|y|> m\}}e^{-\lambda t}p_{t}(y-x) dydt \notag \\
\label{WUSeqlemma44505}=&:I_{n,m}+J_{m}(x).
\end{align}
Similarly to (\ref{WUSeqremark413}), for any fixed $m>0$,
\begin{equation}\label{WUSeqlemma44convofinm}
\lim_{n\to\infty} I_{n,m}=0, \quad \forall  m>0.
\end{equation}
If $(s,x)$ is in the compact set $K$ and $m>0$ is sufficiently large, then
\begin{align*}
J_{m}(x) = & 2M\int_{0}^{\infty}\int_{\{|y'+x|> m\}}e^{-\lambda t}p_{t}(y') dy'dt \\
\le & 2M\int_{0}^{\infty}\int_{\{|y'|> m/2\}}e^{-\lambda t}p_{t}(y') dy'dt.
\end{align*}
By dominated convergence theorem,
 \begin{equation}\label{WUSeqlemma4451}
 \lim_{m \to \infty} \sup_{(s,x)\in K} J_{m}(x)=0.
 \end{equation}
In view of (\ref{WUSeqlemma44505}), (\ref{WUSeqlemma44convofinm}) and (\ref{WUSeqlemma4451}), we can use a simple ``$\epsilon - \delta$'' argument to obtain
\begin{equation}\label{WUSeqlemma446}
\lim_{n\rightarrow\infty}\sup_{(s,x)\in K}R^{\lambda}(|b^{(n)}_{1}-b_{1}|)(s,x)=0.
\end{equation}
It follows from  (\ref{WUSeqremark413}), (\ref{WUSeqlemma445}) and (\ref{WUSeqlemma446}) that
\begin{equation}\label{WUSeqlemma4461}
\lim_{n\rightarrow\infty}\sup_{(s,x)\in K}R^{\lambda}(|b^{(n)}-b|)(s,x)=0.
\end{equation}

We now turn to treat the second term on the right-hand side of (\ref{WUSeqlemma4441}).   For $(s,x)\in K$ and  sufficiently large $m>0$, we have
\begin{align*}
&R^{\lambda}(|b^{(n)}|(1-h_m))(s,x) \\
\le & R^{\lambda}(|b^{(n)}_{1}|(1-h_m))(s,x)+R^{\lambda}(|b^{(n)}_{2}|(1-h_m))(s,x) \\
\le & M\int_{0}^{\infty}\int_{\{|y|> m\}}e^{-\lambda t}p_{t}(y-x) dydt \\
&\quad +\int_{0}^{\infty}e^{-\lambda t} \|p_{t}(\cdot-x)(1-h_m)\|_{L^{p^*}(\Rd)}\|b^{(n)}_{2}(s+t,\cdot)\|_{L^{p}(\Rd)}dt \\
\le & \frac{1}{2}J_m(x) +\int_{0}^{\infty}e^{-\lambda t} \|(1-h_{m/2})p_{t} \|_{L^{p^*}(\Rd)}\|b_{2}(s+t,\cdot)\|_{L^{p}(\Rd)}dt \\
\le & \frac{1}{2}J_m(x) +\int_{0}^{T} \|(1-h_{m/2})p_{t} \|_{L^{p^*}(\Rd)}\|b_{2}(s+t,\cdot)\|_{L^{p}(\Rd)}dt \\
\le & \frac{1}{2}J_m(x) + \|(1-h_{m/2})p_{t} \|_{L^{q^*}([0,T];L^{p^*}(\Rd))}   \|b_{2}\|_{L^q([0,T];L^p(\Rd))}.
\end{align*}
By (\ref{WUSeqlemma4451}) and dominated convergence theorem, we see that
\begin{equation}\label{WUSeqlemma447}
\lim_{m\rightarrow\infty}\sup_{\substack{(s,x)\in K \\ n \in \BN}}R^{\lambda}(|b^{(n)}|(1-h_m))(s,x)=0.
\end{equation}
For any $\epsilon>0$, we can find large enough $m_0>0$ such that
\[
\sup_{\substack{(s,x)\in K \\ n \in \BN}} R^{\lambda}(|b^{(n)}|(1-h_{m_0}))(s,x) <\frac{\epsilon}{6 C}.
\]
By (\ref{WUSeqlemma4442}) and (\ref{WUSeqlemma4461}), there exists $n_0 \in \BN$ such that, for $n \ge n_0$,
\[
C_{n,m_0} \|R^{\lambda}(|b^{(n)}|)\| <\frac{\epsilon}{3} \quad \mbox{and} \quad \sup_{(s,x)\in K}R^{\lambda}(|b^{(n)}-b|)(s,x) <\frac{\epsilon}{3C}. \]
It follows from (\ref{WUSeqlemma4441}) that
\[
\sup_{(s,x)\in K} |R^{\lambda}(B_{n}R^{\lambda})^{k+1}g(s,x)-R^{\lambda}(BR^{\lambda})^{k+1}g(s,x)| <\epsilon, \qquad \forall n \ge n_0,
\]
which shows (\ref{WUSeqlemma443}) for $k+1$.

Similarly, we can show that  (\ref{WUSeqlemma444}) is also true for $k+1.$ Hence Claim \ref{WUSclaim2} is true for any $k \in \BZ_+$. This completes the proof.
\end{proof}

\begin{remark}\label{WUSremarkGlambdab}
By (\ref{WUSeqlemma442}), Proposition \ref{WUSprop32} and Lemma \ref{WUSlemma34}, we have
\begin{align*}
\Vert R^{\lambda}(BR^{\lambda})^{k}(|b|)\Vert
 \leq &  \Vert \nabla R^{\lambda} (|b|)\Vert \big(\kappa_{\lambda}\big)^{k-1}\big(M\lambda^{-1}+N_{\lambda}\|b_2\|_{L^q([0,T];L^p(\Rd))}\big) \\
 =&  \Vert \nabla R^{\lambda} (|b_1+b_2|)\Vert \big(\kappa_{\lambda}\big)^{k-1}\big(M\lambda^{-1}+N_{\lambda}\|b_2\|_{L^q([0,T];L^p(\Rd))}\big) \\
 \le &  \big(\kappa_{\lambda}\big)^{k}\big(M\lambda^{-1}+N_{\lambda}\|b_2\|_{L^q([0,T];L^p(\Rd))}\big).
 \end{align*}
 For $\lambda>\lambda_0$, $G^{\lambda}(|b|):=\sum_{k=0}^{\infty}  R^{\lambda}(BR^{\lambda})^{k}(|b|)$ is thus well-defined and
 \[
 \| G^{\lambda}(|b|) \| \le \sum_{k=0}^{\infty}\| R^{\lambda}(BR^{\lambda})^{k}(|b|)\| <\infty.
 \]

\end{remark}

%%%%%%%%%%%%%%%%%%%%
%  Lemma 4.5
%%%%%%%%%%%%%%%%%%%%

\begin{lemma}\label{WUSlemma45}As $\lambda \to \infty$, $G_{n}^{\lambda}(|b^{(n)}|)(s,x)$ converges to $0$ uniformly in $(s,x) \in \hs$ and $n\in \BN$.
\end{lemma}

\begin{proof}
Let $m,n \in \BN$ and $\lambda>\lambda_0$. Since $|b^{(m)}|$ is bounded, by (\ref{WUSeqdefiforGlambdan1}),
\[
G_{n}^{\lambda}(|b^{(m)}|)=\sum_{k=0}^{\infty}R^{\lambda}(B_{n}R^{\lambda})^{k}(|b^{(m)}|).
\]
Since $|b^{(m)} | \le |b^{(m)}_{1}|+|b^{(m)}_{2}|$, by (\ref{WUSeqlemma438}), Proposition \ref{WUSprop32} and Lemma \ref{WUSlemma34}, for each $k\in \BZ_+$,
\begin{align*}
|(B_{n}R^{\lambda})^{k+1}(|b^{(m)}|)|\leq &\|\nabla R^{\lambda}(|b^{(m)}|)\|(\kappa_{\lambda})^{k}(|b^{(n)}_{1}|+|b^{(n)}_{2}|) \\
\le & (\kappa_{\lambda})^{k+1}(|b^{(n)}_{1}|+|b^{(n)}_{2}|).
\end{align*}
Therefore,
\begin{align}
    G_{n}^{\lambda}(|b^{(m)}|)\le & R^{\lambda} (|b^{(m)}|)+ \textstyle \sum_{k=1}^{\infty}(\kappa_{\lambda})^{k} R^{\lambda} (|b^{(n)}_{1}|+|b^{(n)}_{2}|)  \notag \\
\le & R^{\lambda} (|b^{(m)}_{1}|+|b^{(m)}_{2}|)+ \textstyle \sum_{k=1}^{\infty}(\kappa_{\lambda})^{k}\big(M\lambda^{-1}+N_{\lambda}\|b_2\|_{L^q([0,T];L^p(\Rd))}\big) \notag \\
\le & (M\lambda^{-1}+N_{\lambda}\|b_2\|_{L^q([0,T];L^p(\Rd))})\textstyle \sum_{k=0}^{\infty}(\kappa_{\lambda})^{k} \notag\\
\label{WUSeqlemma451}= & (M\lambda^{-1}+N_{\lambda}\|b_2\|_{L^q([0,T];L^p(\Rd))})(1-\kappa_{\lambda})^{-1}.
\end{align}
Since $N_{\lambda} \downarrow 0$ as $\lambda \to \infty$, it follows that $G_{n}^{\lambda}(|b^{(n)}|)(s,x)$ converges to 0 uniformly in $(s,x) \in \hs$ and $n\in \BN$ as $\lambda \to \infty$.
\end{proof}

We are now ready to prove the local weak existence for the SDE (\ref{WUSeqsect1}).
%%%%%%%%%%%%%%%%%%%%
%  Theorem 4.6
%%%%%%%%%%%%%%%%%%%%
\begin{theorem}\label{WUSthm46}
Let $d\ge2$ and $1<\alpha<2$. Assume Assumption \ref{WUSassumption41}. Then for each $(s,x)\in \mathbb{R}_{+}\times \mathbb{R}^{d}$, there exists a probability space $(\Omega,\mathcal{A},\mathbf{P}),$  on which a non-degenerate $\alpha$-stable process $(S_{t})_{t\geq 0}$ and a c\`adl\`ag process  $(X_{t})_{t\geq 0 }$ are defined, such that (\ref{WUSeqsect1}) is satisfied and
\begin{equation}\label{WUSeqthm45-1}
\mathbf{E}\Big[\int_{0}^{\infty}e^{-\lambda t}g(t+s,X_{t})dt\Big]=G^{\lambda}g(s,x),\qquad \lambda>\lambda_{0},\ g\in \mathcal{B}_{b}(\mathbb{R}_{+}\times \mathbb{R}^{d}),
\end{equation}
where $G^{\lambda}$ is defined in (\ref{WUSeqlemma440}).
\end{theorem}
\begin{proof}
For the construction of  weak solutions to (\ref{WUSeqsect1}), we basically follow the proofs of \cite[Theorem 4.1]{MR3127913}  and \cite[Theorem 3.1]{MR2359059}. We will see that the weak solutions constructed in this way would automatically satisfy (\ref{WUSeqthm45-1}).

Since $b^{(n)}(\cdot,\cdot)$ is bounded and globally Lipschitz in the space variable, for any non-degenerate $\alpha$-stable process $S$ defined on a filtered probability space  $(\Omega,\mathcal{A},\mathbf{P}),$ there exists a strong solution $(X_{t}^{n})_{t\geq 0}$ to the SDE
\[
\begin{cases}dX^n_{t}=dS_{t}+b^{(n)}(s+t, X^n_{t})dt, & \ t\ge 0,\\
   \   X_{0}=x.
    \end{cases}
\]
Therefore,  for each $t \ge 0$,
\begin{equation}\label{WUSeqthm450}
X_{t}^{n}=x+S_{t}+\int_{0}^{t}b^{(n)}(s+u,X_{u}^{n})du \quad \ \mbox{a.s.}.
\end{equation}

Define $Y:=(Y_{t}^{n})_{t \ge 0}$ with $Y_{t}^{n}:= \int_{0}^{t}b^{(n)}(s+u,X_{u}^{n})du$, $t\ge0$, and $Z^{n}:=(X^{n},Y^{n},S).$ Since the remaining proof is rather long, we do it into several steps.

``\textit{Step 1}'': We show that the family $\{Z^{n}:n \in \BN\}$ of random elements in $(D^3, \CD^3)$  is tight.  It suffices to  show that
\begin{equation}\label{WUSeqthm451}
 \lim_{l\rightarrow\infty}\limsup_{n\rightarrow\infty}\bfP\Big(\sup_{0\leq u\leq t}|Z_{u}^{n}|>l\Big)=0,\quad \forall t\geq 0,
\end{equation}
and
\begin{equation}\label{WUSeqthm452}
\limsup_{n\rightarrow\infty}\bfP \big(|Z_{t\wedge(\tau^{n}+r_{n})}^{n}-Z_{t\wedge\tau^{n}}^{n}|>\epsilon\big)=0,\quad  \forall t\geq 0,\ \epsilon>0,
\end{equation}
where each $\tau^{n}$ is an stopping-time with respect to the natural filtration induced by $Z^n$, $n\in\BN$, and $(r_{n})_{n\in \mathbb{N}}$ is a sequence of real numbers with $r_{n}\downarrow 0$ as $n\to \infty$. For $0\leq u\leq t,$ we have
$$
|Z_{u}^{n}|\leq \sqrt{3}\Big(|x|+|S_{u}|+\int_{0}^{u}|b^{(n)}(s+r,X_{r}^{n})|dr\Big),
$$
so
$$
\text{  }\sup_{0\leq u\leq t}|Z_{u}^{n}|\leq \sqrt{3}\Big(|x|+\ \sup_{0\leq u\leq t}|S_{u}|+\int_{0}^{t}|b^{(n)}(s+r,X_{r}^{n})|dr\Big).
$$
Therefore, to get (\ref{WUSeqthm451}), it suffices to show
\begin{equation}\label{WUSeqthm4521}
 \lim_{l\rightarrow\infty}\limsup_{n\rightarrow\infty}\bfP\Big(\int_{0}^{t}|b^{(n)}(s+u,X_{u}^{n})|du>l\Big)=0,\quad \forall t\geq 0.
\end{equation}
By Chebyshev's inequality and Lemma \ref{newlemmasn}, we have
\begin{align}\bfP\Big( \int_{0}^{t}|b^{(n)}(s+u,X_{u}^{n})|du>l\Big)\leq  & l^{-1}\mathbf{E}\Big[\int_{0}^{t}|b^{(n)}(s+u,X_{u}^{n})|du\Big]\notag \\
\le &  l^{-1} e^{\theta t} \mathbf{E}\Big[\int_{0}^{t}e^{-\theta u}|b^{(n)}(s+u,X_{u}^{n})|du\Big] \notag\\
\le & l^{-1} e^{\theta t} \mathbf{E}\Big[\int_{0}^{\infty}e^{-\theta u}|b^{(n)}(s+u,X_{u}^{n})|du \Big]\notag\\
\label{WUSeqthm453}\le &  l^{-1} e^{\theta t} G_{n}^{\theta}(|b^{(n)}|)(s,x),
\end{align}
where $\theta >  \lambda_0 $ is a constant. By (\ref{WUSeqlemma451}), $\sup_{n\in \BN}G_{n}^{\theta}(|b^{(n)}|)(s,x)<\infty$. So (\ref{WUSeqthm4521}) follows from (\ref{WUSeqthm453}). As a consequence, (\ref{WUSeqthm451}) is proved.

By  (\ref{WUSeqthm450}),  we have
$$
|Z_{t\wedge(\tau^{n}+r_{n})}^{n}-Z_{t\wedge\tau^{n}}^{n}|\leq 2\sqrt{3}\Big(|S_{t\wedge(\tau^{n}+r_{n})}-S_{t\wedge\tau^{n}}|+\int_{t\wedge\tau^{n}}^{t\wedge(\tau^{n}+r_{n})}|b^{(n)}(s+u,X_{u}^{n})|du\Big),\
$$
so
\begin{align}
\bfP(|Z_{t\wedge(\tau^{n}+r_{n})}^{n}-Z_{t\wedge\tau^{n}}^{n}|>\epsilon)\leq & \bfP\Big (|S_{t\wedge(\tau^{n}+r_{n})}-S_{t\wedge\tau^{n}}|>\frac{\epsilon}{4\sqrt{3}}\Big) \notag\\
\label{WUSeqthm454}& \quad +\bfP\Big(\int_{t\wedge\tau^{n}}^{t\wedge(\tau^{n}+r_{n})}|b^{(n)}(s+ u,X_{u}^{n})|du>\frac{\epsilon}{4\sqrt{3}}\Big).
\end{align}
Since $S_0=0$ a.s. and $S$ has c\`adl\`ag paths, it follows from strong Markov property that
\begin{equation}\label{WUSeqthm4541}
\limsup_{n \to \infty} \bfP\Big (|S_{t\wedge(\tau^{n}+r_{n})}-S_{t\wedge\tau^{n}}|>\frac{\epsilon}{4\sqrt{3}}\Big)  \le \limsup_{n\rightarrow\infty}\bfP\Big(\sup_{0\leq u\leq r_{n}}|S_{u}|>\frac{\epsilon}{4\sqrt{3}}\Big) =0.
\end{equation}
Again by the strong Markov property,
\begin{align}
& \bfP\Big(\int_{t\wedge\tau^{n}}^{t\wedge(\tau^{n}+r_{n})}|b^{(n)}(s+ u,X_{u}^{n})|du>\frac{\epsilon}{4\sqrt{3}}\Big) \notag \\
\label{WUSeqthm4540} =& \bfE \Big[\bfP\Big(\int_{t\wedge\tau^{n}}^{t\wedge(\tau^{n}+r_{n})}|b^{(n)}(s+u,X_{u}^{n})|du>\frac{\epsilon}{4\sqrt{3}}\ \Big | \ X^n_{t\wedge\tau^{n}}\Big) \Big].
\end{align}
By Chebyshev's inequality and Lemma \ref{newlemmasn},
\begin{align}
& \bfP\Big(\int_{t\wedge\tau^{n}}^{t\wedge(\tau^{n}+r_{n})}|b^{(n)}(s+u,X_{u}^{n})|du>\frac{\epsilon}{4\sqrt{3}}\ \Big | \ X^n_{t\wedge\tau^{n}}\Big) \notag  \\
\le & 4\sqrt{3}\epsilon^{-1}\bfE\Big[\int_{t\wedge\tau^{n}}^{t\wedge\tau^{n}+r_{n}}|b^{(n)}(s+u,X_{u}^{n})|du \ \Big | \ X^n_{t\wedge\tau^{n}} \Big]  \notag \\
\le & 4\sqrt{3}e \epsilon^{-1}\bfE\Big[\int_{t\wedge\tau^{n}}^{t\wedge\tau^{n}+r_{n}}\exp\Big(-\frac{u-t\wedge\tau^{n}}{r_{n}}\Big) |b^{(n)}(s+u,X_{u}^{n})|du \ \Big | \ X^n_{t\wedge\tau^{n}} \Big] \notag \\
\le & 4\sqrt{3}e\epsilon^{-1}G_{n}^{\lambda_n}(|b^{(n)}|)(X^n_{t\wedge\tau^{n}}) \notag \\
\label{WUSeqthm45400}\le&  4\sqrt{3}e\epsilon^{-1}\|G_{n}^{\lambda_n}(|b^{(n)}|)\|,
\end{align}
where $\lambda_n:=1/r_n$. By Lemma \ref{WUSlemma45}, we have $
\lim_{n\to \infty} \|G_{n}^{\lambda_n}(|b^{(n)}|)\| =0$. It follows from (\ref{WUSeqthm4540}) and (\ref{WUSeqthm45400}) that
\begin{equation}\label{WUSeqthm4542}
\limsup_{n\rightarrow\infty}\bfP\Big(\int_{t\wedge\tau^{n}}^{t\wedge(\tau^{n}+r_{n})}|b^{(n)}(s+ u,X_{u}^{n})|du>\frac{\epsilon}{4\sqrt{3}}\Big)=0.
\end{equation}
Combining (\ref{WUSeqthm454}), (\ref{WUSeqthm4541}) and (\ref{WUSeqthm4542}), we obtain (\ref{WUSeqthm452}).

As shown in the proof of \cite[Theorem 3.1]{MR2359059}, we can use the conditions (\ref{WUSeqthm451}) and (\ref{WUSeqthm452}) to find a probability space $(\tilde{\Omega},\tilde{\mathcal{A}},\tilde{\bfP})$
and  processes $\tilde{Z}=(\tilde{X},\tilde{Y},\tilde{S}),\ \tilde{Z}_{n}=(\tilde{X}^{n},\tilde{Y}^{n},\tilde{S}^{n}),\ n=1,2,\cdots,$ defined on it, such that

(i)$\ \ \tilde{Z}_{n}\rightarrow\ \tilde{Z}$ \ $\tilde{\bfP}$-a.s. (as random elements in $(D^3, \CD^3)$) as $n\rightarrow\infty$.

(ii)$\ \tilde{Z}_{n}$ and $Z_{n}$ are identically distributed for each $n\in \mathbb{N}$.

(iii)\ $\tilde{Z}$, $\tilde{Z}_{n}$, $n=1,2,\cdots,$ have  c\`adl\`ag paths.

\noindent In fact, the above three properties hold only for a subsequence $\tilde{Z}_{n_k}$, $k \in \BN$; however, for simplicity, we denote this subsequence still by  $\tilde{Z}_{n}$, $n \in \BN$. It is easy to see that $\tilde{S}^n$ and $\tilde{S}$ are both $\alpha$-stable processes with the characteristic exponent $\psi$.

By (ii) and (\ref{WUSeqthm450}), we have
\begin{equation}\label{WUSeqthm4544}
\tilde{X}_{t}^{n}=x+\tilde{S}^n_{t}+\int_{0}^{t}b^{(n)}(s+u,\tilde{X}_{u}^{n})du \quad \tilde{\bfP}\mbox{-a.s.},  \ \forall  t\geq s .\
\end{equation}
According to (iii) and \cite[~Chap.~3,~Lemma~7.7]{MR838085}, there exists a countable set $I\subset \BR_+$ such that
\begin{equation}\label{WUSeqthm4545}
\tilde{\bfP}(\tilde{Z}_{t-}=\tilde{Z}_{t})=1, \qquad  \forall t \in \BR_+ \setminus I.
\end{equation}
It follows from (i), (\ref{WUSeqthm4545}) and \cite[~Chap.~3,~Prop.~5.2]{MR838085} that
\begin{equation}\label{WUSeqthm4546}
\lim_{n\to\infty}\tilde{X}^n_{t}=\tilde{X}_{t} \quad \mbox{and} \quad \lim_{n\to\infty}\tilde{S}^n_{t}=\tilde{S}_{t} \quad \tilde{\bfP}\mbox{-a.s.}, \qquad \forall t \in \BR_+ \setminus I.
\end{equation}

``\textit{Step 2}'': Next, we show that, for any $f\in \CB_b(\hs)$ and $\lambda>\lambda_0$,
\[
\tilde{\bfE}\Big[\int_0^{\infty}\exp(-\lambda t)f(t+s,\tilde{X}_t)dt\Big]=G^{\lambda}f(s,x),
\]
where $\tilde{\mathbf{E}}[\cdot]$ denotes the expectation taken with respect to the probability measure  $\tilde{\mathbf{P}}$ on $(\tilde{\Omega},\tilde{\mathcal{A}})$ and $G^{\lambda}$ is defined in (\ref{WUSeqlemma440}). We first consider $f \in C_b(\hs)$. By dominated convergence theorem,
\[
\lim_{n\to \infty} \tilde{\bfE}[f(t+s,\tilde{X}^n_t)]=\tilde{\bfE}[f(t+s,\tilde{X}_t)], \qquad \forall t \in \BR_+ \setminus I.
\]
Since $I$ is countable, by dominated convergence and Fubini's theorem, we have
\begin{align}
\lim_{n\to \infty} \tilde{\bfE}\Big[\int_0^{\infty}\exp(-\lambda t)f(t+s,\tilde{X}^n_t)dt\Big]=& \lim_{n\to \infty} \int_0^{\infty}\exp(-\lambda t)\tilde{\bfE}\big[f(t+s,\tilde{X}^n_t)\big]dt \notag \\
=&  \int_0^{\infty}\exp(-\lambda t)\tilde{\bfE}\big[f(t+s,\tilde{X}_t)\big]dt \notag \\
\label{WUSeqthm4547}=&  \tilde{\bfE}\Big[\int_0^{\infty}\exp(-\lambda t)f(t+s,\tilde{X}_t)dt\Big].
\end{align}
From Lemma \ref{newlemmasn} we know that, for any $\lambda>\lambda_0$,
\begin{equation}\label{WUSeqthm4548}
\tilde{\bfE}\Big[\int_0^{\infty}\exp(-\lambda t)f(t+s,\tilde{X}^n_t)dt\Big]=G^{\lambda}_nf(s,x).\end{equation}
Since $G^{\lambda}_nf \to G^{\lambda}f$ locally uniformly as $n\to \infty$ by Lemma \ref{WUSlemma44}, it follows from (\ref{WUSeqthm4547}) and (\ref{WUSeqthm4548}) that
\begin{equation}\label{WUSeqthm4549}
\tilde{\bfE}\Big[\int_0^{\infty}\exp(-\lambda t)f(t+s,\tilde{X}_t)dt\Big]=G^{\lambda}f(s,x), \quad f \in C_b(\hs).
\end{equation}

For any open subset $O $ of $  \hs$, we can find $f_n \in C_b(\hs)$, $n \in \BN$, such that $ 0 \le f_n \uparrow \bfi_O$ as $n\to \infty$.  For $\lambda>\lambda_0$,
\begin{align}
|G^{\lambda}\bfi_O(s,x)-G^{\lambda}f_n(s,x)|=& |G^{\lambda}(\bfi_O-f_n)(s,x)|\notag \\
\label{WUSeqopenseto} \le & \sum_{k=0}^{\infty}|R^{\lambda}(BR^{\lambda})^{k}(\bfi_O-f_n)|(s,x).
\end{align}
For each $k \in \BZ_+$, we have
\begin{equation}\label{WUSeqconvopenseto}
\lim_{n\to \infty}|R^{\lambda}(BR^{\lambda})^{k}(\bfi_O-f_n)|(s,x)=0.
\end{equation}
This can be achieved by applying dominated convergence theorem for $k+1$ times. By (\ref{WUSeqopenseto}), (\ref{WUSeqlemma442}) and  (\ref{WUSeqconvopenseto}), we see that
\begin{equation}\label{WUSeqconvglfn}
\lim_{n \to \infty} |G^{\lambda}(\bfi_O-f_n)|(s,x)=0.
\end{equation}
Since (\ref{WUSeqthm4548}) is true for $f=f_n$, with $n \to \infty$, it follows from monotone convergence theorem that (\ref{WUSeqthm4549}) is also true for $f=\bfi_O$. Now, we can use a  monotone class argument to extend (\ref{WUSeqthm4549}) to all $f \in \CB_b(\hs)$.

Similarly to (\ref{WUSeqconvglfn}), we know that $G^{\lambda}(|b|\wedge k)(s,x)$ goes to $G^{\lambda}(|b|)(s,x)$ as $k \to \infty$. By Remark \ref{WUSremarkGlambdab} and monotone convergence theorem,
\begin{align}
\tilde{\bfE}\Big[\int_0^{\infty}\exp(-\lambda t)|b(t+s,\tilde{X}_t)|dt\Big]=& \lim_{k \to \infty} \tilde{\bfE}\Big[\int_0^{\infty}\exp(-\lambda t)(|b(t+s,\tilde{X}_t)|\wedge k )dt\Big]  \notag \\
=& \lim_{k \to \infty}G^{\lambda}(|b|\wedge k)(s,x)\notag \\
\label{WUSeqthm45491}=& G^{\lambda}(|b|)(s,x)\le \|G^{\lambda}(|b|)\|<\infty.
\end{align}
Therefore, for each $t \ge 0$, we have $\int_0^{t}|b(t+s,\tilde{X}_u)|du<\infty \ \tilde{\mathbf{P}}\mbox{-a.s.}$.

``\textit{Step 3}'': We next show that
\begin{equation}\label{WUSeqthm45495}
\tilde{X}_{t}=x+\tilde{S}_{t}+\int_{0}^{t}b(s+u,\tilde{X}_{u})du    \quad \tilde{\mathbf{P}}\mbox{-a.s.}, \quad \forall t \in \BR_+ \setminus I.
\end{equation}
In view of (\ref{WUSeqthm4544}) and (\ref{WUSeqthm4546}), it suffices to show that, for each $t \ge 0$,
\begin{equation}\label{WUSeqthm455}
\int_{0}^{t}b^{(n)}(s+u,\tilde{X}_{u}^{n})du \to \int_{0}^{t}b(s+u,\tilde{X}_{u})du \quad \mbox{in probability as } n \to \infty.
\end{equation}
For any $\delta>0$ and $\lambda > \lambda_0$, by Chebyshev's inequality,
\begin{align}
&\tilde{\bfP}\Big(\big|\int_{0}^{t}b^{(n)}(s+u,\tilde{X}_{u}^{n})du - \int_{0}^{t}b(s+u,\tilde{X}_{u})du\big| >3 \delta\Big) \notag \\
 \le & \tilde{\bfP}\Big(\big|\int_{0}^{t}b^{(k)}(s+u,\tilde{X}_{u}^{n})du - \int_{0}^{t}b^{(k)}(s+u,\tilde{X}_{u})du\big| > \delta\Big) \notag \\
 & \ + \tilde{\bfP}\Big(\big|\int_{0}^{t}(b^{(n)}-b^{(k)})(s+u,\tilde{X}_{u}^{n})du\big| > \delta\Big)  +\tilde{\bfP}\Big(\big|\int_{0}^{t}(b-b^{(k)})(s+u,\tilde{X}_{u})du\big| > \delta\Big) \notag \\
\le & \delta^{-1}\tilde{\bfE}\Big[\big|\int_{0}^{t}b^{(k)}(s+u,\tilde{X}_{u}^{n})du - \int_{0}^{t}b^{(k)}(s+u,\tilde{X}_{u})du\big|\Big] \notag \\
& \ + \delta^{-1}\tilde{\bfE}\Big[\int_{0}^{t}|b^{(n)}-b^{(k)}|(s+u,\tilde{X}_{u}^{n})du \Big] + \delta^{-1}\tilde{\bfE}\Big[\int_{0}^{t}|b-b^{(k)}|(s+u,\tilde{X}_{u})du \Big] \notag \\
\le & \delta^{-1}e^{\lambda t} \bigg( e^{-\lambda t}\tilde{\bfE}\Big[\int_{0}^{t}\big|b^{(k)}(s+u,\tilde{X}_{u}^{n}) - b^{(k)}(s+u,\tilde{X}_{u})\big|du\Big] \notag \\
& \ + \tilde{\bfE}\Big[\int_{0}^{t}e^{-\lambda u}|b^{(n)}-b^{(k)}|(s+u,\tilde{X}_{u}^{n})du \Big]  +\tilde{\bfE}\Big[\int_{0}^{t}e^{-\lambda u}|b-b^{(k)}|(s+u,\tilde{X}_{u})du \Big]\bigg) \notag \\
\le &  \delta^{-1}\int_{0}^{t}\tilde{\bfE}\big[|b^{(k)}(s+u,\tilde{X}_{u}^{n}) - b^{(k)}(s+u,\tilde{X}_{u})|\big]du + e^{\lambda t}\delta^{-1}G^{\lambda}_n(|b^{(n)}-b^{(k)}|)(s,x) \notag \\
\label{WUSeqthm456} & \quad +e^{\lambda t}\delta^{-1}G^{\lambda}(|b-b^{(k)}|)(s,x),
\end{align}
where  $k\in \BN$ will be determined later. At the moment, we assume the following claim is true.
\begin{claim}\label{WUSclaimthm45}
$\lim_{n,k\to\infty}G^{\lambda}_n (|b^{(n)}-b^{(k)}|) (s,x)=\lim_{k\to \infty}G^{\lambda}(|b-b^{(k)}|)(s,x)=0.$
\end{claim}
\noindent The proof of this claim will be given in ``\textit{Step 4}''. According to Claim \ref{WUSclaimthm45}, for any $\epsilon>0$, we can choose $k_0$ large enough such that, for any $n\ge k_0$,
\begin{equation}\label{WUSeqthm4565}
G^{\lambda}(|b-b^{(k_0)}|)(s,x)\le \epsilon \delta e^{-\lambda t} /3\quad \mbox{and} \quad G^{\lambda}_n(|b^{(n)}-b^{(k_0)}|)(s,x)\le \epsilon \delta e^{-\lambda t} /3.
\end{equation}
Noting that  $b^{(k_0)}$ is bounded and globally Lipschitz in the space variable, by dominated convergence theorem,
\begin{equation}\label{WUSeqthm457}
\lim_{n \to \infty}\int_{0}^{t}\tilde{\bfE}\big[|b^{(k_0)}(s+u,\tilde{X}_{u}^{n}) -b^{(k_0)}(s+u,\tilde{X}_{u})|\big]du=0.
\end{equation}
Therefore, there exists $n_0\ge k_0$ such that for $n \ge n_0$,
\begin{equation}\label{WUSeqthm458}
\int_{0}^{t}\tilde{\bfE}\big[|b^{(k_0)}(s+u,\tilde{X}_{u}^{n}) -b^{(k_0)}(s+u,\tilde{X}_{u})|\big]du \le \epsilon \delta /3.
\end{equation}
Combining (\ref{WUSeqthm456}), (\ref{WUSeqthm4565}) and (\ref{WUSeqthm458}) yields
\[
\tilde{\bfP}\Big(\big|\int_{0}^{t}b^{(n)}(s+u,\tilde{X}_{u}^{n})du - \int_{0}^{t}b(s+u,\tilde{X}_{u})du\big| >3 \delta\Big) \le \epsilon, \quad \forall n\ge n_0,
\]
which implies (\ref{WUSeqthm455}).

``\textit{Step 4}'': In this step, we prove Claim \ref{WUSclaimthm45}. Let $\lambda>\lambda_0$. Since \[|b^{(n)}-b^{(k)}|\le |b^{(n)}_{1}-b^{(k)}_{1}|+ |b^{(n)}_{2}-b^{(k)}_{2}|,\] by (\ref{WUSeqthm4548}), we have
\begin{equation}\label{WUSeqthm46step41}
G^{\lambda}_n(|b^{(n)}-b^{(k)}|) \le G^{\lambda}_n(|b^{(n)}_{1}-b^{(k)}_{1}|)+ G^{\lambda}_n(|b^{(n)}_{2}-b^{(k)}_{2}|).
\end{equation}
It follows from Proposition \ref{WUSprop32} and Remark \ref{WUSremark42} that
\[
\|\nabla R^{\lambda}(|b^{(n)}_{2}-b^{(k)}_{2}|)\|\le M_{\lambda} \|b^{(n)}_{2}-b^{(k)}_{2}\|_{L^q([0,T];L^p(\Rd))}.
\]
By (\ref{WUSeqlemma439}), we have, for each $i \in \BN$,
\begin{align*}
&\|R^{\lambda}(B_{n}R^{\lambda})^{i}(|b^{(n)}_{2}-b^{(k)}_{2}|)\| \\
 \le&  M_{\lambda}\Vert b^{(n)}_{2}-b^{(k)}_{2}\Vert_{L^q([0,T];L^p(\Rd))}(\kappa_{\lambda})^{i-1}(M\lambda^{-1}+N_{\lambda}\|b_2\|_{L^q([0,T];L^p(\Rd))}).
\end{align*}
By Proposition \ref{WUSprop32} and (\ref{WUSeqremark413}),
\begin{align*}
& \|G^{\lambda}_n(|b^{(n)}_{2}-b^{(k)}_{2}|)\| \\
\le & \|R^{\lambda}(|b^{(n)}_{2}-b^{(k)}_{2}|)\|+ \sum_{i=1}^{\infty}\|R^{\lambda}(B_{n}R^{\lambda})^{i}(|b^{(n)}_{2}-b^{(k)}_{2}|)\| \\
\le& N _{\lambda}\Vert b^{(n)}_{2}-b^{(k)}_{2}\Vert_{L^q([0,T];L^p(\Rd))}
 +C_1 M_{\lambda}\Vert b^{(n)}_{2}-b^{(k)}_{2}\Vert_{L^q([0,T];L^p(\Rd))},
\end{align*}
where $C_1:=(1-\kappa_{\lambda})^{-1}(M\lambda^{-1}+N_{\lambda}\|b_2\|_{L^q([0,T];L^p(\Rd))})$ is a constant.
It follows  that $G^{\lambda}_n(|b^{(n)}_{2}-b^{(k)}_{2}|)$ converges uniformly to 0 as $n,k \to \infty$.

By (\ref{WUSeqthm46step41}), to show that $G^{\lambda}_n(|b^{(n)}-b^{(k)}|)(s,x)$ goes to 0 as $n,k\to \infty$, it suffices to show that
\begin{equation}\label{WUSeqthm4587}
G^{\lambda}_n(|b^{(n)}_{1}-b^{(k)}_{1}|) \to 0 \quad \mbox{locally uniformly as } n,k\to\infty.
\end{equation}
By (\ref{WUSeqlemma436}), (\ref{WUSeqlemma439}) and Lemma \ref{WUSlemma34}, we have, for each $i \in \BZ_+$,
\begin{equation}\label{WUSeqthm4588}
\|R^{\lambda}(B_{n}R^{\lambda})^{i}(|b^{(n)}_{1}-b^{(k)}_{1}|)\|  \le 2L_{\lambda}M (\kappa_{\lambda})^{i-1}(M\lambda^{-1}+N_{\lambda}\|b_2\|_{L^q([0,T];L^p(\Rd))})
\end{equation}
and
\begin{equation}\label{WUSeqthm4589}
\|\nabla R^{\lambda}(B_{n}R^{\lambda})^{i}(|b^{(n)}_{1}-b^{(k)}_{1}|)\|  \le 2L_{\lambda}M (\kappa_{\lambda})^{i}.
\end{equation}

Next, we show by induction that, for any $i \in \BZ_+$ and compact $K \subset \hs$,
\begin{equation}\label{WUSeqthm459}
\lim_{n,k \to \infty}\sup_{(t,y)\in K} R^{\lambda}(B_{n}R^{\lambda})^{i}(|b^{(n)}_{1}-b^{(k)}_{1}|) (t,y)=0
\end{equation}
and
\begin{equation}\label{WUSeqthm4591}
\lim_{n,k \to \infty}\sup_{(t,y)\in K} |\nabla R^{\lambda}(B_{n}R^{\lambda})^{i}(|b^{(n)}_{1}-b^{(k)}_{1}|) |(t,y)=0
\end{equation}
For $i=0$, by (\ref{WUSeqlemma446}), it is easy to see that
\[
\lim_{n,k \to \infty}\sup_{(t,y)\in K}R^{\lambda}(|b^{(n)}_{1}-b^{(k)}_{1}|)(t,y) = 0.
\]
Very similarly, we also have
\[
\lim_{n,k \to \infty}\sup_{(t,y)\in K}|\nabla R^{\lambda}(|b^{(n)}_{1}-b^{(k)}_{1}|)|(t,y) = 0.
\]
Suppose that (\ref{WUSeqthm459}) and (\ref{WUSeqthm4591}) are true for $i$. For $m>0$ let $A_m$ and $h_m$ be as in (\ref{WUSdefiofAm}) and (\ref{WUSdefiofhm}), respectively.
Set
\[
C_{n,k,m}:=\sup_{(t,y)\in [0,T]\times A_m}|\nabla R^{\lambda}(B_{n}R^{\lambda})^{i}(|b^{(n)}_{1}-b^{(k)}_{1}|)|(t,y).
\]
By induction hypothesis, $\lim_{n,k\to\infty}C_{n,k,m}=0$ for any $m>0$. It follows from (\ref{WUSeqthm4589}) that
\begin{align}
& |R^{\lambda}(B_{n}R^{\lambda})^{i+1}(|b^{(n)}_{1}-b^{(k)}_{1}|)|(t,y) \notag \\
=& \big|R^{\lambda}\big(b^{(n)}(h_m+1-h_m)\cdot \nabla R^{\lambda}(B_{n}R^{\lambda})^{i}(|b^{(n)}_{1}-b^{(k)}_{1}|)\big)\big|(t,y) \notag \\
\le & \big|R^{\lambda}\big(b^{(n)}h_m \cdot \nabla R^{\lambda}(B_{n}R^{\lambda})^{i}(|b^{(n)}_{1}-b^{(k)}_{1}|)\big)\big|(t,y) \notag \\
& \qquad + \big|R^{\lambda}\big(b^{(n)}(1-h_m)\cdot \nabla R^{\lambda}(B_{n}R^{\lambda})^{i}(|b^{(n)}_{1}-b^{(k)}_{1}|)\big)\big|(t,y)  \notag \\
\label{WUSeqthm4592} \le &  C_{n,k,m} \|R^{\lambda}(|b^{(n)}|)\|+  2L_{\lambda}M (\kappa_{\lambda})^{i} |R^{\lambda}(|b^{(n)}|(1-h_m))|(t,y).
\end{align}
Combining (\ref{WUSeqsuprlambdabn}), (\ref{WUSeqlemma447}) and (\ref{WUSeqthm4592}), we obtain (\ref{WUSeqthm459}) for $i+1$. The claim (\ref{WUSeqthm4591}) for $i+1$ can be similarly proved. Therefore, (\ref{WUSeqthm459}) and (\ref{WUSeqthm4591}) are true for any $i \in \BZ_+$.

By (\ref{WUSeqdefiforGlambdan1}), (\ref{WUSeqthm4588}) and (\ref{WUSeqthm459}), we see that (\ref{WUSeqthm4587}) holds.  As a consequence,
\[
\lim_{n,k\to\infty}G^{\lambda}_n(|b^{(n)}-b^{(k)}|)(s,x)=0.
\]
With a very similar argument as above, we conclude that
\[
\lim_{k\to \infty}G^{\lambda}(|b-b^{(k)}|)(s,x)=0.
\]
Thus Claim \ref{WUSclaimthm45} is proved.

``\textit{Step 5}'': Since $I$ is countable and $\tilde{X}$, $\tilde{S}$ and $\int_0^t b(u+s,\tilde{X}_u)du$ all have c\`adl\`ag paths, (\ref{WUSeqthm45495}) must hold for all $t \ge 0$. This completes the proof.
\end{proof}

%%%%%%%%%%%%%%%%%%%%
%  Corollary 4.7
%%%%%%%%%%%%%%%%%%%%
\begin{corollary}\label{WUScoro47} Assume the same assumptions as in Theorem \ref{WUSthm46}. For each $(s,x) \in \hs$, let $X^{s,x}=(X^{s,x}_{t})_{t \ge 0}$ be the solution of (\ref{WUSeqsect1}) that we constructed in Theorem \ref{WUSthm46}. Define the process $Y^{s,x}=(Y^{s,x}_{t})_{t \ge 0}$ by
\[
Y^{s,x}_t=\begin{cases}x, & 0\le t \le s, \\
X^{s,x}_{t-s}, & t>s.\end{cases}
\]
Let $\mathbf{P}^{s,x}$ be the probability measure on the path space $(D,\CD)$ induced by $Y^{s,x}$.  Then $\mathbf{P}^{s,x}$ is a solution to the martingale problem for $L_{t}=A+b(t,\cdot)\cdot \nabla$ starting from $(s,x)$. Moreover,  the family of measures $\{ \mathbf{P}^{s,x}: \ (s,x) \in \hs\}$ is measurable, that is, $\mathbf{P}^{s,x}(A)$ is measurable in $(s,x) $ for every $A\in \CD$.
\end{corollary}
\begin{proof}
A simple application of It\^{o}'s formula leads to the fact that $\mathbf{P}^{s,x}$ is a solution to the martingale problem for $L_{t}$ starting from $(s,x)$.

Let $X=(X_{t})_{t \ge 0}$ be the canonical process  on $(D,\CD)$. Let $\mathbf{E}^{s,x}[\cdot]$ denote the expectation taken with respect to the measure $\mathbf{P}^{s,x}$ on $(D,\CD)$. It follows from (\ref{WUSeqthm45-1}) that, for any $\lambda>\lambda_0$ and $g \in \CB_b(\hs)$,
\begin{equation}\label{WUSeqcoro47}
\mathbf{E}^{s,x}\Big[\int_{s}^{\infty} e^{-\lambda (t-s)}g(t,X_{t})dt\Big]=G^{\lambda}g(s,x)= \sum_{i=0}^{\infty}R^{\lambda}(BR^{\lambda})^{i}g(s,x) .
\end{equation}

We next show that the family  $\{ \mathbf{P}^{s,x}: \ (s,x) \in \hs\}$ is measurable. Let $\varphi \in C_b(\Rd)$ and $t > 0$. Define $\tilde{\varphi}_{n}:\hs \to \BR$ as follows:
\[
\tilde{\varphi}_n(u,y):=\begin{cases}0, & \ u < t,\\
   \  \varphi(y)\rho_n(u-t), & \ u \ge t,
    \end{cases}
\]
where $\rho_n$ is a mollifying sequence on $\BR$ with $\rho_n(u)=\rho_n(-u)$, $u\in \BR$. By (\ref{WUSeqcoro47}), for $s\in [0,t)$, $x\in\Rd$ and $\lambda>\lambda_{0}$,
\[
 \mathbf{E}^{s,x}\Big[\int_{s}^{\infty}e^{-\lambda(u-s)}\tilde{\varphi}_n(u,X_{u})du\Big]=\sum_{i=0}^{\infty}R^{\lambda}(BR^{\lambda})^{i}\tilde{\varphi}_n(s,x) .
\]
It follows that $\mathbf{E}^{s,x}[\int_{s}^{\infty}e^{-\lambda(u-s)}\tilde{\varphi}_n(u,X_{u})du]$ is measurable in $(s,x)$. By dominated convergence theorem and noting that $X_t$ is right-continuous, we get
\begin{align*}
\lim_{n\to \infty}\mathbf{E}^{s,x}\Big[\int_{s}^{\infty}e^{-\lambda(u-s)}\tilde{\varphi}_n(u,X_{u})du\Big]=& \lim_{n\to\infty}\mathbf{E}^{s,x}\Big[\int_{t}^{\infty}e^{-\lambda(u-s)}\varphi(X_u)\rho_n(u-t)du\Big]\\
=& \mathbf{E}^{s,x}[2^{-1} e^{-\lambda(t-s)}\varphi(X_{t})]
\end{align*}
for all $(s,x) \in [0,t)\times \Rd$, which implies that $\mathbf{E}^{s,x}[\varphi(X_{t})]$ is measurable in $(s,x) \in [0,t)\times \Rd$. If $s\ge t$, then $\mathbf{E}^{s,x}[\varphi(X_t)]=\varphi(x)$. Thus $\mathbf{E}^{s,x}[\varphi(X_{t})]$ is measurable in $(s,x) \in \hs$.

Similarly, for $0 \le r_{1} \le \cdots \le r_{l} $ and $g_{1}, \cdots, g_{l} \in C_{b}(\mathbb{R}^{d})$ with $l \in \BN$, one can show that $\mathbf{E}^{s,x}\big[\prod_{j=1}^{l}g_{j}(X_{r_{j}})\big]$ is measurable in $(s,x) \in \hs$. Now, the assertion follows by a monotone class argument.
\end{proof}

%%%%%%%%%%%%%%%%%%%%
%  New Lemma
%%%%%%%%%%%%%%%%%%%%

The following lemma is analog to \cite[Theorem~6.1.3]{MR2190038}, which plays an important role in showing the uniqueness of solutions to martingale problems. With this lemma, we can use the standard argument to show that  multi-dimensional distributions of solutions to the martingale problem for $L_t$ are unique,  provided that  one-dimensional distributions of those are so. Recall that $X=(X_t)_{t \ge 0}$ is the canonical process defined on the path space $(D,\mathcal{D})$ and $(\CF_t)_{t \ge 0}$ is the filtration generated by $(X_t)_{t \ge 0}$.

\begin{lemma}\label{WUSnewlemmamgt}Suppose that the probability measure $\mathbf{Q}^{s,x}$ on $(D=D([0,\infty);\mathbb{R}^{d}),\mathcal{D})$ satisfies: \\
(i) $\mathbf{Q}^{s,x}$ is a solution to the martingale problem for $L_{t}$ starting from $(s,x)$, where $L_t=A+b(t,\cdot)\cdot \nabla$; \\
(ii) $\mathbf{E}_{\mathbf{Q}^{s,x}}\Big[\int_{s}^{\infty} e^{-\lambda (t-s)}|b(t,X_{t})|dt\Big] <\infty.$ \\
For a given $t\ge s$, we denote by $Q_{\omega}(A)=Q(\omega,A): \Omega \times \CF_t \to [0,1]$ the regular conditional distribution of $\mathbf{Q}^{s,x}$ given $\mathcal{F}_{t}$. Then there exists a $\mathbf{Q}^{s,x}$-null set $N\in
\mathcal{F}_{t}$ such that, for each $\omega \notin N$,   $Q_{\omega}$ solves the martingale problem for $L_{t}$ starting from $(t, \omega(t))$ and
\[
\mathbf{E}_{Q_{\omega}}\Big[\int_{t}^{\infty} e^{-\lambda (u-s)}|b(u,X_{u})|du\Big] <\infty.
\]
\end{lemma}

\begin{proof}
We follow the proof of \cite[Theorem~6.1.3]{MR2190038}. Let $\big\{f_{n}:f_{n}\in C_{0}^{\infty}(\mathbb{R}^{d}), \ n \in \BN \big\}$ be dense in $C_{0}^{\infty}(\mathbb{R}^{d})$. By \cite[Theorem 1.2.10]{MR2190038}, for each $f_{n}$, there exists $N_{n}\in \mathcal{F}_{t}$ such that $\mathbf{Q}^{s,x}(N_{n})=0$ and, for all $\omega \notin N_{n}$,
\[
    M^{f_{n}}_u:=f_{n}(X_{u})-f_{n}(X_{t})-\int^{u}_{t}L_rf_{n}(r,X_{r})dr, \quad u \ge t, \]
is an $\mathcal{F}_{u}$-martingale after time $t$ with respect to $Q_{\omega}$.

By (ii), we have
\[
\int_D \mathbf{E}_{Q_{\omega}}\Big[\int_{t}^{\infty} e^{-\lambda (u-s)}|b(u,X_{u})|du\Big] \mathbf{Q}^{s,x}(d\omega)<\infty,
\]
so there exists $\mathbf{Q}^{s,x}$-null set $N_{0} \in \CF_{t}$ such that
\begin{equation}\label{WUSeqnewlemmaapp}
\mathbf{E}_{Q_{\omega}}\Big[\int_{t}^{\infty} e^{-\lambda (u-s)}|b(u,X_{u})|du\Big] <\infty \quad\text{for all} \  \omega \notin N_{0}.\end{equation}
Let $N:=\cup_{n \ge 0} N_{n}.$

We now fix $\omega \in \Omega \setminus N$.  For any $f \in C^{\infty}_{0}(\mathbb{R}^{d})$, we can find $f_{n_{k}}$ such that $f_{n_{k}} \to f$ in $C^{\infty}_{0}(\mathbb{R}^{d})$ as $k \to \infty$. In view of (\ref{WUSeqnewlemmaapp}), we have  for all $u \ge t$,
\[
    M^{f_{n_{k}}}_{u } \to M^{f}_{u } \qquad Q_{\omega}\mbox{-a.s.} \quad \mbox{as} \quad k \to \infty.\]
By (\ref{WUSeqnewlemmaapp}), dominated convergence theorem and the martingale property of $M^{f_{n_{k}}}$, we see that $M^{f}$ is also an $\mathcal{F}_{u}$-martingale after $t$ with respect to $Q_{\omega}$. Thus $Q_{\omega}$ solves the martingale problem for $L_{t}$ starting from $(t,\omega(t))$.
\end{proof}

%%%%%%%%%%%%%%%%%%%%
%  Prop. 4.8
%%%%%%%%%%%%%%%%%%%%

\begin{proposition}\label{WUSprop48}
Let $d\ge2$ and $1<\alpha<2$. Assume that the $\alpha$-stable process $S$ is non-degenerate and Assumption \ref{WUSassumption41} is satisfied. Then for each $(s,x)\in \hs$, solutions of the SDE (\ref{WUSeqsect1}) are weakly unique.
\end{proposition}
\begin{proof}Our proof is adapted from the proof of  \cite[Proposition 5.1]{MR1964949}. Consider an arbitrary weak solution of (\ref{WUSeqsect1}), that is, a non-degenerate $\alpha$-stable process $S=(S_t)_{t \ge 0}$ and a c\`adl\`ag process $Y=(Y_t)_{t \ge 0}$ that are both defined on some  probability space $(\Omega, \CA, \mathbf{Q})$ and are such that
\[
Y_{t}=x+S_{t}+\int_0^tb(s+u, Y_{u})du \quad \mbox{a.s.},  \quad \forall t\ge 0.
\]
Define a measurable map $\Phi: \Omega \to D$ by
\[
\Omega \ni \omega \mapsto \Phi(\omega) \in D, \quad \mbox{where} \quad \Phi(\omega)(t):=\begin{cases} x, &t \le s,\\ Y_{t-s}(\omega) , &t>s. \end{cases}
\]
Let  $\mathbf{Q}^{s,x}$ be the image measure  of $\mathbf{Q}$ on $(D,\CD)$ under the map $\Phi$. Then it is routine to check that $\mathbf{Q}^{s,x}$ is a solution to the martingale problem for $L_{t}=A+b(t,\cdot) \cdot \nabla$ starting from $(s,x)$. To show that weak uniqueness  for (\ref{WUSeqsect1}) holds, it suffices to prove $\mathbf{Q}^{s,x}=\mathbf{P}^{s,x}$ on $(D,\CD)$, where $\mathbf{P}^{s,x}$ is defined  in Corollary \ref{WUScoro47}.

Let $(\mathcal{M}_{t})_{t \ge0}$ be the usual augmentation of $(\CF_{t})_{t \ge0}$ with respect to the measure $\mathbf{Q}^{s,x}$. Define a sequence of $\CF_{t+}$-stopping times
\[\sigma_{n}:=\inf\{t\ge s: \int_{s}^{t}|b(u,X_{u})|du > n\}, \quad n \in \BN,  \]
and let
\[
 \tau_{n}:=\sigma_{n}\land n  \quad \mbox{for} \ n \in \BN \ \ \mbox{with} \ \ n \ge s.
\]
Clearly, $\sigma_n$ and $\tau_n$ are also $\CM_{t}$-stopping times. According to the condition (\ref{WUSeqsect12}), we have $\tau_n \to \infty$ $\ \mathbf{Q}^{s,x}$-a.s..

For each fixed $\omega \in D$, it follows from \cite[Lemma 6.1.1]{MR2190038} that there is a unique probability measure $\delta_{\omega} \bigotimes_{ \tau_{n}(\omega)}\mathbf{P}^{ \tau_{n}(\omega), \omega(\tau_{n})}$ on $(D, \CD)$ such that
\[
\textstyle \delta_{\omega} \bigotimes_{ \tau_{n}(\omega)}\mathbf{P}^{ \tau_{n}(\omega), \omega(\tau_{n})}\big(X_t=\omega(t), \ 0\le t \le \tau_n(\omega)\big)=1 \]
and
\[
\textstyle \delta_{\omega} \bigotimes_{ \tau_{n}(\omega)}\mathbf{P}^{ \tau_{n}(\omega), \omega(\tau_{n})}(A)=\mathbf{P}^{ \tau_{n}(\omega),\omega(\tau_{n})}(A), \quad A \in \CF^{\tau_{n}(\omega)},
\]
where $\CF^{t}:=\sigma(X(r): r\ge t)$ for $t \ge0$. In view of Corollary \ref{WUScoro47}, we can easily check that $\delta_{(\cdot)} \bigotimes_{ \tau_{n}(\cdot)}\mathbf{P}^{ \tau_{n}(\cdot), (\cdot)(\tau_{n})}$  is a probability kernel from $(D, \CM_{\tau_n})$ to $(D, \CD)$. For details, the reader is referred to the proof of \cite[Theorem~6.1.2]{MR2190038}. Thus $\delta_{(\cdot)} \bigotimes_{ \tau_{n}(\cdot)}\mathbf{P}^{ \tau_{n}(\cdot), (\cdot)(\tau_{n})}$ induces a  probability measure $\mathbf{Q}^{s,x}_{n}$ on $(D, \CD)$ with
\[
    \mathbf{Q}^{s,x}_{n}(A)= \int_{D}\textstyle\delta_{\omega} \bigotimes_{ \tau_{n}(\omega)}\mathbf{P}^{ \tau_{n}(\omega), \omega(\tau_{n})}(A)\mathbf{Q}^{s,x}(d\omega), \quad   A \in \mathcal{D}.\]
For each $\lambda>\lambda_0$,
\begin{align}&  \mathbf{E}_{\mathbf{Q}^{s,x}_{n}}\Big[\int_{s}^{\infty} e^{-\lambda (t-s)}|b(t,X_{t})|dt\Big] \notag \\
 =& \mathbf{E}_{\mathbf{Q}^{s,x}}\Big[\int_{s}^{\tau_{n}} e^{-\lambda (t-s)}|b(t,X_{t})|dt\Big]\notag \\
&\quad +\mathbf{E}_{\mathbf{Q}^{s,x}}\Big[e^{-\lambda (\tau_{n}-s)}\mathbf{E}_{\mathbf{P}^{\tau_{n},X_{\tau_{n}}}}\big[\int_{\tau_{n}}^{\infty} e^{-\lambda (t-\tau_{n})}|b(t,X_{t})|dt\big]\Big]\notag\\
\label{eqthm2infty1}\le & n+\mathbf{E}_{\mathbf{Q}^{s,x}}\Big[e^{-\lambda (\tau_{n}-s)}\mathbf{E}_{\mathbf{P}^{\tau_{n},X_{\tau_{n}}}}\big[\int_{\tau_{n}}^{\infty} e^{-\lambda (t-\tau_{n})}|b(t,X_{t})|dt\big]\Big]. %\mathbf{E}_{\mathbf{Q}^{s,x}}\Big[G^{\lambda}(|b|)(\tau_{n},X_{\tau_{n}})\Big].
 \end{align}
Just as in (\ref{WUSeqthm45491}), we have
\begin{equation}\label{eqthm2infty2}
\mathbf{E}_{\mathbf{P}^{\tau_{n},X_{\tau_{n}}}}\Big[\int_{\tau_{n}}^{\infty} e^{-\lambda (t-\tau_{n})}|b(t,X_{t})|dt\Big]= G^{\lambda}(|b|)(\tau_{n},X_{\tau_{n}}) \le  \|G^{\lambda}(|b|)\| < \infty.
\end{equation}
It follows from (\ref{eqthm2infty1}) and (\ref{eqthm2infty2}) that
\begin{equation}\label{WUSeqprop483}
\mathbf{E}_{\mathbf{Q}^{s,x}_{n}}\Big[\int_{s}^{\infty} e^{-\lambda (t-s)}|b(t,X_{t})|dt\Big] <\infty.
\end{equation}

Next, we proceed to show $\mathbf{Q}_n^{s,x}=\mathbf{P}^{s,x}$. For $f \in C_{0}^{\infty}(\mathbb{R}^{d})$, set
\[
    M^{f}_t:=f(X_{t})-f(X_{s})-\int^{t}_{s}L_uf(u,X_{u})du, \quad t \ge s.
\]
Then  $M^{f}:=(M^{f}_t)_{t\ge s}$ is an $\mathcal{F}_{t}$-martingale after time $s$ with respect to the measure $\mathbf{Q}_{n}^{s,x}$. To see this, let $s \le t_1 \le t_2$ and $A \in \CF_{t_1}$. Then
\begin{align*}
& \mathbf{E}_{\mathbf{Q}^{s,x}_{n}}[M^{f}_{t_2};A] \\
=& \int_{\{\tau_{n} \le t_1\}} \mathbf{E}_{\mathbf{P}^{ \tau_{n}(\omega), \omega(\tau_{n})}}[M^{f}_{t_2};A]\mathbf{Q}^{s,x}(d\omega)+\int_{\{t_1 < \tau_{n} \le t_2\}}\textstyle M^{f}_{\tau_n}(\omega) \bfi_{A}(\omega)\mathbf{Q}^{s,x}(d\omega)  \\
& \quad + \int_{\{ \tau_{n} > t_2\}}\textstyle M^{f}_{t_2}(\omega) \bfi_{A}(\omega)\mathbf{Q}^{s,x}(d\omega) \\
=& \int_{ \{\tau_{n} \le t_1\}}\textstyle \mathbf{E}_{\mathbf{P}^{ \tau_{n}(\omega), \omega(\tau_{n})}}[M^{f}_{t_1};A]\mathbf{Q}^{s,x}(d\omega)+\mathbf{E}_{\mathbf{Q}^{s,x}}\big[M^{f}_{t_2 \land \tau_n};A\cap \{\tau_{n} > t_1\}\big].
\end{align*}
Since $A\cap \{\tau_{n} > t_1\} \in \CM_{t_1 \land \tau_n}$, by optional sampling theorem,
\begin{align*}
 \mathbf{E}_{\mathbf{Q}^{s,x}}\big[M^{f}_{t_2 \land \tau_n};A\cap \{\tau_{n} > t_1\}\big]=& \mathbf{E}_{\mathbf{Q}^{s,x}}\big[M^{f}_{t_1 \land \tau_n};A\cap \{\tau_{n} > t_1\}\big] \\
 =& \mathbf{E}_{\mathbf{Q}^{s,x}}\big[M^{f}_{t_1};A\cap \{\tau_{n} > t_1\}\big].
\end{align*}
So
\begin{align*}
 & \mathbf{E}_{\mathbf{Q}^{s,x}_{n}}[M^{f}_{t_2};A] \\
=&  \int_{ \{\tau_{n} \le t_1\}}\textstyle \mathbf{E}_{\mathbf{P}^{ \tau_{n}(\omega), \omega(\tau_{n})}}[M^{f}_{t_1};A]\mathbf{Q}^{s,x}(d\omega)+ \mathbf{E}_{\mathbf{Q}^{s,x}}\big[M^{f}_{t_1};A\cap \{\tau_{n} > t_1\}\big] \\
=& \int_{ \{\tau_{n} \le t_1\}} \mathbf{E}_{\mathbf{P}^{ \tau_{n}(\omega), \omega(\tau_{n})}}[M^{f}_{t_1};A]\mathbf{Q}^{s,x}(d\omega) +\int_{\{ \tau_{n}>t_1 \}} M^{f}_{t_1}(\omega) \bfi_{A}(\omega)\mathbf{Q}^{s,x}(d\omega) \\
=& \mathbf{E}_{\mathbf{Q}^{s,x}_{n}}[M^{f}_{t_1};A].
\end{align*}
This shows that $M^{f}$ is an $\mathcal{F}_{t}$-martingale after time $s$ with respect to $\mathbf{Q}_{n}^{s,x}$. It follows from (\ref{WUSeqprop483}) and \cite[Theorem~4.2.1]{MR2190038} that,
for any $f\in C^{1,2}_b(\hs)$,
\begin{align*}& f(t,X_{t})-f(s,X_{s}) -\int_{s}^{t}(\frac{\partial f}{\partial u}+L_{u}f)(u,X_{u})du
\end{align*}
is an $\CF_t$-martingale after time $s$ with respect to $\mathbf{Q}^{s,x}_{n}$. As a consequence,
\[
    \mathbf{E}_{\mathbf{Q}^{s,x}_{n}}[f(t,X_{t})]-f(s,x)= \mathbf{E}_{\mathbf{Q}^{s,x}_{n}}\Big[\int_{s}^{t}(\frac{\partial f}{\partial u}+L_{u}f)(u,X_{u})du\Big]. \]

We now define a linear functional $V_{n}^{\lambda}$ as follows. For any measurable function $f$ on $\hs$ with
\[
\mathbf{E}_{\mathbf{Q}^{s,x}_{n}}\Big[\int_{s}^{\infty} e^{-\lambda (t-s)}|f(t,X_{t})|dt\Big]<\infty,
\]
let
\[
V_{n}^{\lambda}f:=\mathbf{E}_{\mathbf{Q}^{s,x}_{n}}\Big[\int_{s}^{\infty} e^{-\lambda (t-s)}f(t,X_{t})dt\Big]. \]
By the same argument that we used to obtain (\ref{neweqvn0}), we get
\begin{equation}\label{neweqvn}
    V_{n}^{\lambda}g=R^{\lambda}g(s,x)+V_{n}^{\lambda}BR^{\lambda}g, \quad g \in \CB_b(\hs).
\end{equation}
where $BR^{\lambda}$ is defined by (\ref{defiofBR}). Since, by (\ref{WUSeqprop483}), $V_{n}^{\lambda}(|b|)< \infty$, we can use the equality (\ref{neweqvn}) and dominated convergence theorem to get
\[
    V_{n}^{\lambda}BR^{\lambda}g=R^{\lambda}BR^{\lambda}g(s,x)+V_{n}^{\lambda}(BR^{\lambda})^{2}g  ,\quad g \in \CB_b(\hs),\]
which implies
\[
V_{n}^{\lambda}g=R^{\lambda}g(s,x)+ R^{\lambda}BR^{\lambda}g(s,x)+V_{n}^{\lambda}(BR^{\lambda})^{2}g  ,\quad g \in \CB_b(\hs). \]
After a simple induction, we obtain
\[
    V_{n}^{\lambda}g=\sum_{i=0}^{k}R^{\lambda}(BR^{\lambda})^{i}g(s,x)+ V_{n}^{\lambda}(BR^{\lambda})^{k+1}g ,\quad g \in \CB_b(\hs). \]
For $\lambda>\lambda_0$, by (\ref{WUSeqlemma437}) and (\ref{WUSeqlemma441}), we get
\[
            |V_{n}^{\lambda}(BR^{\lambda})^{k+1}g|\le  \|\nabla R^{\lambda}(BR^{\lambda})^{k}g\| V_{n}^{\lambda}(|b|) \to 0 \quad \mbox{as} \ \ k \to \infty.
            \]
It follows from Lemma \ref{WUSlemma44} and (\ref{WUSeqcoro47}) that, for any $\lambda>\lambda_0$ and $g \in \CB_b(\hs)$,
\begin{align*}
     \mathbf{E}_{\mathbf{Q}^{s,x}_{n}}\Big[\int_{s}^{\infty} e^{-\lambda (t-s)}g(t,X_{t})dt\Big]=V_{n}^{\lambda}g=&\sum_{i=0}^{\infty}R^{\lambda}(BR^{\lambda})^{i}g(s,x) \\
     =& \mathbf{E}_{\mathbf{P}^{s,x}}\Big[\int_{s}^{\infty} e^{-\lambda (t-s)}g(t,X_{t})dt\Big].\end{align*}
By the uniqueness of the Laplace transform, we have
\[
\mathbf{E}_{\mathbf{Q}^{s,x}_{n}}[f(X_{t})]=\mathbf{E}_{\mathbf{P}^{s,x}}[f(X_{t})],\quad \forall f\in C_{b}(\mathbb{R}^{d}),\ t \ge s.\]
This means that one-dimensional distributions of $\mathbf{Q}^{s,x}_{n}$ and $\mathbf{P}^{s,x}$ are the same. By (\ref{WUSeqprop483}), Lemma \ref{WUSnewlemmamgt} and the argument in the proof of \cite[Theorem~6.2.3]{MR2190038}, we conclude that multiple-dimensional distributions of $\mathbf{Q}^{s,x}_{n}$ and $\mathbf{P}^{s,x}_{n}$ also coincide, that is,   $\mathbf{Q}^{s,x}_{n}=\mathbf{P}^{s,x}$ on $(D, \CD)$.

Note that $\tau_n \to \infty$ $\ \mathbf{Q}^{s,x}$-a.s.. With the same reason, we have $\tau_n \to \infty$ $\ \mathbf{P}^{s,x}$-a.s.. Since $\mathbf{Q}^{s,x}_{n}$ and $\mathbf{Q}^{s,x}$ coincide before $\tau_n$, it follows from dominated convergence that,
for $0 \le r_{1} \le \cdots \le r_{l} $ and $g_{1}, \cdots, g_{l} \in C_{b}(\mathbb{R}^{d})$ with $l \in \BN$,
\begin{align*}
 \mathbf{E}_{\mathbf{Q}^{s,x}}\big[\textstyle\prod_{j=1}^{l}g_{j}(X_{r_{j}})\big] =& \lim_{n\to\infty}\mathbf{E}_{\mathbf{Q}^{s,x}}\big[\textstyle\prod_{j=1}^{l}g_{j}(X_{r_{j}\land \tau_n})\big] \\
  =& \lim_{n\to\infty}\mathbf{E}_{\mathbf{Q}_{n}^{s,x}}\big[\textstyle\prod_{j=1}^{l}g_{j}(X_{r_{j}\land \tau_n})\big] \\
  = & \lim_{n\to\infty}\mathbf{E}_{\mathbf{P}^{s,x}}\big[\textstyle\prod_{j=1}^{l}g_{j}(X_{r_{j}\land \tau_n})\big] \\
  =& \ \mathbf{E}_{\mathbf{P}^{s,x}}\big[\textstyle\prod_{j=1}^{l}g_{j}(X_{r_{j}})\big].
\end{align*}
Thus $\mathbf{Q}^{s,x}=\mathbf{P}^{s,x}$ on $(D, \CD)$. This completes the proof.
\end{proof}

\section{Existence and Uniqueness of Weak Solutions: Global Case}
In this section we study the general case and prove Theorem \ref{WUSmainthm}. In contrast to Theorem \ref{WUSthm46} and Proposition \ref{WUSprop48}, the main task here is to remove the restriction that $supp(b)\subset[0,T]\times \mathbb{R}^{d}$, which was assumed in the previous section.

\subsection*{Proof of Theorem \ref{WUSmainthm}}
``\textit{Existence}": For each $n \in \BN$, consider the drift $a_n: \hs \to \Rd$ defined by
\begin{equation}\label{WUSeqmainthm0}
a_n(t,\cdot):=\begin{cases}b(t,\cdot), & t \le n, \\ 0, & \mbox{otherwise}. \end{cases}
\end{equation}
According to Theorem \ref{WUSthm46}, there exist a c\`adl\`ag process $X^n=(X^n_t)_{t \ge 0}$ and a non-degenerate $\alpha$-stable process $S^n=(S^n_t)_{t \ge 0}$ with characteristic exponent $\psi$ that are both defined on some probability space $(\Omega_n, \CA_n, \bfP_n)$ such that
\begin{equation}\label{WUSeqmainthm}
X^n_{t}=x+S^n_{t}+\int_0^ta_n(s+u, X_{u})du \quad \mbox{a.s.},  \quad \forall t\ge 0.
\end{equation}
Define $  \Phi_n: (\Omega_n, \CA_n) \to (D,\CD)$ by
\[
\Omega_n \ni \omega \mapsto \Phi_n(\omega) \in D, \quad \mbox{where} \quad \Phi_n(\omega)(t):=\begin{cases} x, &t \le s,\\ X^n_{t-s}(\omega) , &t>s. \end{cases}
\]
Consider the measure $\mathbf{Q}_n^{s,x}$ on $(D,\CD)$ defined as $\mathbf{Q}_n^{s,x}:=\bfP_n \circ  (\Phi_n)^{-1}$, that is,  $\mathbf{Q}_n^{s,x}$ is the image measure of $\mathbf{P}_n$ under the map $\Phi_n$. Since $supp(a_{n})\subset[0,T]\times \mathbb{R}^{d}$, by the local weak uniqueness for (\ref{WUSeqmainthm}) that we have shown in Proposition \ref{WUSprop48}, the measures $\mathbf{Q}_n^{s,x}$, $n\in \BN$, must be consistent, that is, $\mathbf{Q}_{n+1}^{s,x}|_{\CF_{n}}=\mathbf{Q}_n^{s,x}|_{\CF_{n}}$ for all $n \in \BN$.  It follows from the projective limit theorem (see, e.g., \cite[Corollary~6.15]{MR1876169}) that there exists a  probability measure $\mathbf{Q}^{s,x}$ on $(D,\CD)$ such that $\mathbf{Q}^{s,x}|_{\CF_{n}}=\mathbf{Q}_n^{s,x}|_{\CF_{n}}$ for all $n \in \BN$. Let $X=(X_t)_{t \ge 0}$ be the canonical process defined on $(D,\mathcal{D})$. For any $\xi \in \Rd$ and $t \ge 0$, by choosing $n \in \BN$ such that $n \ge t+s$, we obtain
\begin{align*}
 \bfE_{\mathbf{Q}^{s,x}}\Big[e^{i\xi \cdot (X_{s+t}-x -\int_0^tb(s+u,X_{s+u})du)}\Big]
= & \bfE_{\mathbf{Q}_n^{s,x}}\Big[e^{i\xi \cdot (X_{s+t}-x -\int_0^tb(s+u,X_{s+u})du)}\Big] \\
=& \bfE_{\mathbf{P}_n}\Big[e^{i\xi \cdot (X^n_{t}-x -\int_0^{t}b(s+u,X^n_u)du)}\Big] \\
=& \bfE_{\mathbf{P}_n}\Big[e^{i\xi \cdot (X^n_{t}-x -\int_0^{t}a_n(s+u,X^n_u)du)}\Big] \\
=& \bfE_{\mathbf{P}_n}\Big[e^{i\xi \cdot S^n_{t}}\Big]=e^{-t\psi(\xi)},
\end{align*}
that is, under the measure $\mathbf{Q}^{s,x}$, the process
\[
S_t:=X_{s+t}-x -\int_0^tb(s+u,X_{s+u})du, \quad t \ge 0,
\]
is an $\alpha$-stable process with characteristic exponent $\psi$. Define $\tilde{X}:=(\tilde{X}_{t})_{t \ge 0}$ with $\tilde{X}_{t}:=X_{t+s}$, $t \ge 0$. Then $\tilde{X}$ satisfies
\[
\tilde{X}_{t}=x+S_{t}+\int_0^tb(s+u, \tilde{X}_{u})du \quad \mathbf{Q}^{s,x}\mbox{-a.s.},  \quad \forall t\ge 0.
\]
Thus $\tilde{X}$ is a weak solution to the SDE (\ref{WUSeqsect1}).

``\textit{Uniqueness}": Consider an arbitrary weak solution of (\ref{WUSeqsect1}), that is, a non-degenerate $\alpha$-stable process $S=(S_t)_{t \ge 0}$ and a c\`adl\`ag process $Y=(Y_t)_{t \ge 0}$ that are both defined on some probability space $(\Omega, \CA, \mathbf{Q})$ such that
\[
Y_{t}=x+S_{t}+\int_0^tb(s+u, Y_{u})du \quad \mbox{a.s.},  \quad \forall t\ge 0.
\]
For $k \in \BN$ with $k\ge s$, let $a_k$ be as in (\ref{WUSeqmainthm0}), and define $Y^{k}_t:=Y_t$ for $0\le t \le k-s$ and
\[
Y^{k}_t:=Y_{k-s}+S_t-S_{k-s}, \quad t >k-s.
\]
Thus
\[
Y^{k}_{t}=x+S_{t}+\int_0^ta_k(s+u, Y^{k}_{u})du \quad \mbox{a.s.},  \quad \forall t\ge 0.
\]
By Proposition \ref{WUSprop48}, the law of $Y^k:=(Y^{k}_t)_{t\ge0}$ is uniquely determined. Since $Y_t=Y^k_t$ for $t \le k-s$, the law of the process $Y$ is  uniquely determined at least up to time $k-s$. With $k \to \infty$, we see that the law of $X$ is completely and uniquely determined.   \qed
\bibliographystyle{amsplain}

\end{document}